\pdfoutput=1
\documentclass[11pt]{article}

\usepackage[a4paper,margin=1in]{geometry}
\usepackage[T1]{fontenc}
\usepackage{lmodern}
\usepackage{xcolor}
\usepackage{amsmath}
\usepackage{amssymb}
\usepackage{amsthm}
\usepackage{multirow}
\usepackage{tikz-cd}
\usepackage{subcaption}
\usepackage{algorithm}
\usepackage{algpseudocode}
\usepackage{Macros/mydef}
\usepackage{bbm}
\usepackage{ulem}
\usepackage{todonotes}
\usepackage{xspace}
\usepackage{comment}
\usepackage[colorlinks=true,linkcolor=blue,citecolor=blue,urlcolor=blue]{hyperref}
\usepackage[nameinlink,capitalise]{cleveref}

\newtheorem{theorem}{Theorem}[section]
\newtheorem{proposition}[theorem]{Proposition}

\theoremstyle{remark}
\newtheorem{remark}[theorem]{Remark}

\crefname{hypothesis}{Hypothesis}{Hypotheses}
\crefname{fact}{Fact}{Facts}

\newcommand{\mred}[1]{#1}

\newenvironment{keyword}{\par\smallskip\noindent\textbf{Keywords: }}{\par\smallskip}
\newenvironment{MSCcodes}{\par\noindent\textbf{MSC classifications: }}{\par\medskip}

\title{A second-order structure- and positivity-preserving convex limiting method for the Vlasov equations}
\author{Katharina Kormann\thanks{Department of Mathematics, Ruhr-Universität Bochum, Universitätsstraße 150, D-44801 Bochum, Germany. Email: \texttt{k.kormann@rub.de}.}
  \and Murtazo Nazarov\thanks{Division of Scientific Computing, Department of Information Technology, Uppsala University, Uppsala 751 05 Sweden. Email: \texttt{murtazo.nazarov@uu.se}.}
  \and Junjie Wen\thanks{Corresponding author. Division of Scientific Computing, Department of Information Technology, Uppsala University, Uppsala 751 05 Sweden. Email: \texttt{junjie.wen@it.uu.se}.}}
\date{}
\usepackage{amsopn}

\begin{document}

\maketitle
\begin{center}\small
Authors are listed alphabetically. This work was supported by the Swedish Research Council (VR), grants 2021-04620 and 2021-05095, and by the German Research Foundation (DFG), project 530709913.
\end{center}

%
%
%

\begin{abstract}
  In this paper, we introduce a novel second-order, positivity-preserving finite element method for the Vlasov equations using a convex limiting algorithm. The method employs \mred{strong-stability-preserving (SSP) Runge--Kutta} time integration and a tensor-product construction of the phase-space mesh for efficient high-dimensional implementations. The convex limiting algorithm combines the robust first-order positivity-preserving graph viscosity approach with high-order residual-based viscosity stabilization to obtain a high-order positivity-preserving scheme. Both novel first-order and high-order methods applicable to high-dimensional problems such as the Vlasov system are presented. In addition, we propose a divergence-cleaning technique for Maxwell’s equations to ensure that the divergence constraints of the electromagnetic fields are satisfied. Numerical experiments are provided to demonstrate the accuracy and robustness of the proposed methods.
\end{abstract}



\begin{keyword}
  Vlasov equations; Structure-preserving; Discrete maximum principle; Residual-based viscosity; Finite element methods.
\end{keyword}

\begin{MSCcodes}
  65M60, 35L65, 35Q83
\end{MSCcodes}

\section{Introduction}
The Vlasov equation provides a kinetic description of the evolution of plasma under external and self-consistent electromagnetic fields.  Obtaining accurate numerical solutions to the Vlasov equations is a challenging task, as it is a high-dimensional advection equation nonlinearly coupled with Maxwell's equations. In particular, the curse of dimensionality poses severe computational challenges, especially in phase spaces beyond the classical 1D1V and 1D2V settings. Since the Vlasov equation is purely advection-dominated, its finite element approximation requires high-order accuracy, robust stabilization techniques, and computational efficiency. In addition to numerical performance, it is also crucial to preserve the physical properties of the Vlasov--Maxwell systems, including the conservation of total mass, energy, and momentum; the divergence constraints on the electromagnetic fields, as well as the positivity- and asymptotic-preserving properties of the distribution function. Preserving all of these properties simultaneously is highly challenging, but various techniques can be employed to recover selected properties depending on the requirements of a given problem.

Positivity preserving is an essential property of the Vlasov equation, as it describes the evolution of the distribution function of particles, which must remain non-negative, and this property is extremely important for collisional plasma simulation, see \eg  \cite{MR2100514}. Positivity preserving for the Vlasov equation can be achieved using various reconstruction methods and techniques, with one of the earliest investigations reported in \cite{MR1852326}. The semi-Lagrangian positivity-preserving (SL-PP) framework is widely used for the Vlasov--Poisson equations, and it is often combined with other spatial discretizations, including discontinuous Galerkin (DG), finite difference (FD), and spectral volume (SV) methods, see \eg \cite{MR2806222, MR2843721, MR4329985, MR3227274, zhang2025positivitypreservinghighordersemilagrangianspectral}. Recently, a positivity-preserving framework based on the active flux method was proposed for the Vlasov--Poisson equations \cite{MR4844786}, in which a flux limiter is incorporated to enforce positivity for 1D1V Vlasov--Poisson equations. In addition, several other positivity-preserving approaches have been developed for the Vlasov equations, including the symmetrically weighted square-root formulation \cite{MR4773545}. 

High-order discontinuous Galerkin (DG) methods, in particular those based on the framework of \cite{Zhang_Shu_2011}, provide a systematic approach for constructing positivity-preserving schemes for conservation laws and have been successfully applied to Vlasov--Poisson \cite{MR2806222,MR2843721} and Vlasov--Boltzmann transport equations \cite{MR2833491}. However, despite their flexibility and high-order accuracy, DG methods typically involve a significantly larger number of degrees of freedom per element compared to continuous finite element methods, leading to increased memory usage and computational cost. This limitation becomes especially pronounced in high-dimensional phase space, where fully resolved simulations rapidly become prohibitively expensive. Recent developments such as sparse-grid and low-rank DG methods, see \eg \cite{Schnake_et_al_2024, MR3566908, MR4710847, MR4728768, MR4925872} and references therein, aim to mitigate these costs, but robust positivity-preserving and structure-preserving formulations in these settings are still in their infancy \cite{Ye_Loureiro_2024}. Moreover, these methods require certain structural properties of the solution and the computational domain which can be found in many application but not all.

To the best of our knowledge, there currently exists no continuous finite element approximation of the Vlasov system for which positivity preservation has been rigorously established. This lack of robust positivity-preserving frameworks has limited the applicability of continuous finite element methods for kinetic plasma simulations, despite their attractive properties, including a more compact representation and improved memory efficiency.

High-order positivity-preserving schemes typically rely on three essential components: a robust first-order positivity-preserving method, a high-order (preferably entropy-convergent) discretization, and a suitable conservative limiting algorithm. While these components have been successfully combined in DG and finite difference frameworks, their extension to continuous finite element methods has remained challenging. In particular, the absence of a robust first-order positivity-preserving scheme, analogous to the Lax--Friedrichs method, has been a key obstacle in transferring these techniques to the finite element setting.

Recent developments in \cite{Guermond_Nazarov_2014, Guermond_Popov_2017} introduced a graph-based viscosity formulation, which can be interpreted as a finite element analogue of the Lax--Friedrichs scheme, and established robust first-order positivity-preserving methods for general nonlinear conservation laws within the continuous finite element framework. For the high-order component, artificial viscosity can be constructed using either entropy viscosity (EV) \cite{MR2787948} or residual-based viscosity (RV) \cite{Nazarov_2013}. The RV approach is particularly attractive, as convergence to the entropy solution has been established for general scalar conservation laws \cite{Nazarov_2013}. A limiting procedure, such as the classical flux-corrected transport (FCT) method \cite{Boris_books_JCP_1973,Zalesak_1979}, or \textit{convex limiting} \cite{Guermond_etal_2018}, can then be employed to obtain high-order positivity-preserving schemes; see, for example, \cite{Guermond_etal_2014}.

The schemes in \cite{Guermond_Nazarov_2014, Guermond_Popov_2017} employ lumped mass matrices, making them computationally feasible for high-dimensional problems. However, extending residual-based viscosity methods such as \cite{Nazarov_2013} to high-dimensional settings remains challenging. In particular, the residual is typically constructed using a finite element projection, which introduces significant computational overhead in high dimensions. Moreover, time derivatives are often approximated using backward differentiation formulas, requiring storage of multiple solution states from previous time steps. One possible alternative is to replace the time derivative with a spatial approximation of the flux term; see, for example, \cite{MR3998292,MR4208958}. However, even with this modification, the associated projection problem can remain computationally expensive in high-dimensional settings.

In this work, we address these limitations by proposing a residual construction that is specifically tailored to tensor-product finite element spaces and avoids the need to store the solutions at multiple time levels. This feature is particularly important for high-dimensional phase-space simulations, where memory efficiency is a critical concern.

For the Vlasov--Maxwell system, Gauss’s law can be shown to hold at the continuous level as a consequence of the charge continuity equation. Consequently, it also holds at the discrete level provided that Maxwell’s equations are discretized consistently with the Vlasov equation, as discussed in~\cite{Kraus_Kormann_Morrison_Sonnendruecker_2017,KORMANN2021109890,MR4804196}. However, several techniques employed in practice, including artificial viscosity, mass lumping, and convex limiting, may violate the discrete continuity equation and thereby destroy the discrete validity of Gauss’s law.

To restore Gauss’s law, di\-ver\-gen\-ce-cleaning techniques may be applied. Such approaches are widely used in magnetohydrodynamics, where maintaining the di\-ver\-gen\-ce-free constraint of the magnetic field is nontrivial. Divergence cleaning explicitly removes spurious divergence errors from the computed fields, thereby eliminating nonphysical components, see~\cite[Sec.~3.2.1]{MR4456197}. The same strategy can be applied in the Vlasov--Maxwell setting to enforce Gauss’s law.

In this manuscript, we introduce a positivity-preserving finite element framework for approximating the Vlasov equations. Our goal is to preserve positivity while maintaining high-order accuracy. Two types of artificial viscosity are proposed: a low-order viscosity that strictly enforces the maximum principle, and a high-order viscosity based on the residual of the Vlasov equation. Convex limiting combines these schemes, yielding a method that preserves the maximum principle and retains high-order accuracy. 

We employ a tensor-product construction of the finite element spaces, as in \cite{MR4321466,MR4945433}, and apply the schemes in up to four dimensions. The proposed residual formulation is designed to be compatible with this tensor-product structure and to reduce computational overhead in high dimensions. For the Vlasov--Maxwell equations, we introduce a divergence-cleaning technique to enforce Gauss’s law in the weak form. 

The key contribution of this work is to demonstrate that continuous finite element methods can be equipped with a provably positivity-preserving and structure-preserving framework for Vlasov-type problems. This establishes finite element methods as a viable and computationally efficient alternative to DG methods in high-dimensional settings, where memory efficiency and scalability are critical. \mred{The present realization uses piecewise multilinear elements and is therefore second-order accurate in phase space; this restriction is not intrinsic to the underlying low-order/high-order limiting strategy. One possible extension uses the closest-neighbor reconstruction techniques developed in \cite{Abgrall_2017, Guermond_et_al_2024}.} We also note that, for certain plasma benchmarks, second-order methods can exhibit competitive or even superior performance due to improved robustness, as illustrated in \cite[Figure~3(b)]{MR4907486}.

This manuscript is organized as follows: Section~\ref{section:Pre} is a review, introducing the Vlasov equations and the tensor-product finite element method. In Section~\ref{section:pp}, we present a positivity-preserving finite element framework for advection equations, and in this section we propose new approaches for computing the residuals. Section~\ref{section:FEMMaxwell} describes finite element schemes for Maxwell's equations, including a divergence-cleaning technique. All numerical results are presented in Section~\ref{section:experiment}. Section~\ref{section:conclu} concludes with a summary of the work.

\section{Preliminaries}\label{section:Pre}
In this section, we introduce the governing equations of Vla\-sov--Maxwell, and a tensor-product finite element method for approximating the Vlasov equation.

\subsection{Governing equations}
The Vlasov equation with appropriate boundary conditions and initial data is given by
\begin{equation}\label{eq:vm}
  \partial_t f_s +\pmb{v}\cdot\nabla_{\pmb{x}}f_s+\frac{q_s}{m_s}(\pmb{E}+\pmb{v}\times\pmb{B})\cdot\nabla_{\pmb{v}}f_s=0,
\end{equation}
where $f_s(\pmb{x},\pmb{v},t)$ is the distribution function of the particle species $s$ with charge $q_s$ and mass $m_s$, at position $\pmb{x}\in\mathbb{R}^{ {d_x}}$ with velocity $\pmb{v}\in\mathbb{R}^{ {d_v}}$ at time $t\geq0$, here  {$d_x$ and $d_v$ are} the space dimensions.
The electric field $\pmb{E}(\pmb{x},t)$ and magnetic field $\pmb{B}(\pmb{x},t)$ are characterized by the following Maxwell's equations 
\begin{equation}
  \begin{aligned}
    \frac{1}{c^2}\partial_t \pmb{E} &= \nabla_{\pmb{x}}\times \pmb{B} - \mu_0\pmb{J},&\qquad\nabla_{\pmb{x}}\cdot \pmb{E} &= \frac{\rho}{\epsilon_0},\\
    \partial_t \pmb{B} &= -\nabla_{\pmb{x}}\times \pmb{E}, &\qquad\nabla_{\pmb{x}}\cdot \pmb{B} &= 0,
  \end{aligned}
\end{equation}
where $c = \frac{1}{\sqrt{\epsilon_0\mu_0}}$ is the speed of light in vacuum, $\epsilon_0$ is the permittivity, and $\mu_0$ is the permeability. Throughout the remainder of this manuscript, these constants are assumed to have unit value for simplicity. Here, $\pmb{J}(\pmb{x},t)$ is the current density and $\rho(\pmb{x},t)$ is the charge density, defined as
\begin{equation}
  \pmb{J}(\pmb{x},t) = \sum_{s}q_s\int_{\mathbb{R}^{ {d_v}}} \pmb{v} f_s(\pmb{x},\pmb{v},t) \, {\rm d}\pmb{v},\qquad \rho(\pmb{x},t) = \sum_s q_s\int_{\mathbb{R}^{ {d_v}}} f_s(\pmb{x},\pmb{v},t) \, {\rm d}\pmb{v}.
\end{equation}

The above system is called Vlasov--Maxwell, and one simplification of it is the Vlasov--Poisson system, which is given by disregarding the magnetic field $\pmb{B}$ in \eqref{eq:vm}, and computing the electric field $\pmb{E}$ from the Poisson equation:
\begin{equation}
  \pmb{E} = -\nabla_{\pmb{x}}\Phi,\qquad-\nabla_{\pmb{x}}^2\Phi = \rho.
\end{equation}

In this manuscript, we focus on a single-species simulation and normalize the charge density by
\[
  \rho = \int_{\mathbb{R}^{ {d_v}}} f\, {\rm d}\pmb{v} - \rho_0,
\]
where $\rho_0$ is the constant background such that $\int_{\mathbb{R}^{ {d_x}}} \rho\, {\rm d}\pmb{x} = 0$. In addition, we assume $q_s/m_s\equiv1$.

\subsection{Finite element approximations}
We use continuous piecewise bilinear finite elements for approximating the Vlasov equation. The phase space $\Omega$ is the product of the configuration space $\Omega_{\pmb{x}}$ and velocity space $\Omega_{\pmb{v}}$. We define the function spaces of piecewise bilinear functions for $\Omega_{\pmb{x}}$ and $\Omega_{\pmb{v}}$, respectively, and then construct the space for phase space using tensor products.

Let us define the appropriate function spaces and notations for the domain $\Omega_{\pmb{z}}$, where $\pmb{z}$ can be $\pmb{x}$ or $\pmb{v}$.
Let $\calT_{\pmb{z}}$ be a grid mesh of $\Omega_{\pmb{z}}$, which consists of a finite union of cuboids
$K_{\pmb{z}} := \prod^{d_z}_{l=1}[z_{l}^-, z_{l}^+]$, where $d_z$ is the dimension of $\Omega_{\pmb{z}}$, $z_{l}^-, z_{l}^+\in \mathbb{R}$ and $z_{l}^- < z_{l}^+$, such that $\overline{\Omega}_{\pmb{z}} = \bigcup_{K_{\pmb{z}}\in \calT_{{\pmb{z}}}} \overline{K}_{\pmb{z}}$, where $\overline{\Omega}_{\pmb{z}}$ and $\overline{K}_{\pmb{z}}$ denote the closure of $\Omega_{\pmb{z}}$ and $K_{\pmb{z}}$ respectively. We use the following bilinear finite element
\begin{equation}
  \polQ_1({\pmb{z}})={\rm span}\left\{\prod^{d_z}_{l=1}z_i^{\alpha_l}:0\le\alpha_l\le 1\ {\rm for}\ l=1,\ldots,d_z\right\},\notag
\end{equation}
and define the following function space of piecewise bilinear functions
\begin{equation*}
  \calV_{\pmb{z}} := \{w(\pmb{z})\in \calC^0(\overline\Omega_{\pmb{z}}):\, w(\bT_K(\widehat{\pmb{z}}))\in \polQ_{ {1}}(\widehat{K}),\forall K\in \calT_{\pmb{z}} \},
\end{equation*}
where $\bT_K(\widehat{\pmb{z}})$ denotes an affine mapping from the reference cuboid $\widehat{K}$ to the physical cell $K\in \calT_{{\pmb{z}}}$.

Let $\{\phi_i\}_{i=1}^{N_x}$ and $\{\varphi_j\}_{j=1}^{N_v}$ be the sets of basis functions of $\calV_{\pmb{x}}$ and $\calV_{\pmb{v}}$, respectively, where $N_x={\rm dim}(\calV_{\pmb{x}})$ and $N_v={\rm dim}(\calV_{\pmb{v}})$.
We define the mesh $\calT_h:=\calT_{\pmb{x}}\times\calT_{\pmb{v}}$ on the phase space and the function space $\calV_h$ on $\calT_h$. Denote the basis set of $\calV_h$ by $\{\psi_i\}_{i=1}^N$, where $N={\rm dim}(\calV_h)$. Since \mred{$\calV:=\calV_{{\pmb{x}}}\otimes\calV_{{\pmb{v}}}$}, $N=N_xN_v$ and all the basis functions in $\calV_h$ can be decomposed to be the products of the basis functions in $\calV_{\pmb{x}}$ and $\calV_{\pmb{v}}$.
We assign the following graphic index to the basis function in $\calV_h$ such that: $\psi_{(i-1)N_v+j}=\phi_i\varphi_j$, $i=1,\ldots,N_x$ and $j=1,\ldots,N_v$. We further define $\calI(l)$ be the set of all nodes within the support of $\psi_{l}$, $l=1,\ldots,N$, and let ${\rm card}(\calI(l))$ denote the number of nodes in $\mathcal{I}(l)$. 

\subsection{Tensor-product finite element method}
We seek the finite element approximation $f_h(\pmb{x}, \pmb{v}, t) \in \calC^1([0,T]; \calV_h)$ such that
\begin{equation}
  (\partial_t f_h+\pmb{v}\cdot\nabla_{\pmb{x}}f_h+(\pmb{E}_h+\pmb{v}\times\pmb{B}_h)\cdot\nabla_{\pmb{v}}f_h,\psi_i)=0,\qquad i=1,\ldots,N,\notag
\end{equation}
where $\pmb{E}_h$ and $\pmb{B}_h$ are the approximations of $\pmb{E}$ and $\pmb{B}$, respectively, and will be defined later. The above formula is equivalent to
\begin{equation}\label{eq:GFEM:ts}
  \begin{aligned}
    \bigg(\partial_t\sum_{j=1}^N f_j\psi_j,\psi_i\bigg)+\bigg(\pmb{v}\cdot\nabla_{\pmb{x}}\sum_{j=1}^N f_j\psi_j,\psi_i\bigg)+\bigg(\pmb{E}_h\cdot\nabla_{\pmb{v}}\sum_{j=1}^N f_j\psi_j,\psi_i\bigg)\\
    +\bigg(\pmb{B}_h\cdot\Big(\nabla_{\pmb{v}}\sum_{j=1}^N f_j\psi_j\times\pmb{v}\Big),\psi_i\bigg)=0,\qquad i=1,\ldots,N,
  \end{aligned}
\end{equation}
where $f_i:=f_i(t)$ denotes the degree of freedom of $f_h$ corresponding to the $i$-th node, with $i = 1, \ldots, N$.
Upon defining $\polf=(f_1,\cdots,f_N)^{\mathsf{T}}$, we obtain the following system from \eqref{eq:GFEM:ts}
\begin{equation}\label{eq:system}
  (\polM^{{\pmb{x}}}\otimes \polM^{{\pmb{v}}})\dot{\polf} + 
  \sum_{l=1}^3 
  \Big( 
  \polA^{\pmb{x},l}\otimes\polC^{\pmb{v},l}
  + 
  \polC^{\pmb{x},l}(\pmb{E}_h)\otimes \polA^{\pmb{v},l} + 
  \polC^{\pmb{x},l}(\pmb{B}_h)\otimes \polG^{\pmb{v},l}
  \Big)  
  \polf= 0,
\end{equation}
where the operators are defined as follows
\begin{align*}
  \polA_{ij}^{{\pmb{x}},l} &= \int_{\Omega_{\pmb{x}}}(\partial_{x_l}\phi_j)\phi_i\ {\rm d}\pmb{x}, \quad
                             \polC_{ij}^{{\pmb{x}},l}(\pmb{F}_h) = \int_{\Omega_{\pmb{x}}}\pmb{F}_l\phi_j\phi_i\ {\rm d}\pmb{x}, 
                             \quad l = 1,2,3,\\
  \polM_{ij}^{{\pmb{x}}} &= \int_{\Omega_{\pmb{x}}}\phi_j\phi_i\ {\rm d}\pmb{x}, \quad
                           i,j=1,\ldots,N_x,
\end{align*}
where $\pmb{F}_h := \{\pmb{E}_h, \pmb{B}_h\}$, and 
\begin{align*}
  \polA_{ij}^{{\pmb{v}},l} &= \int_{\Omega_{\pmb{v}}}(\partial_{v_l}\varphi_j)\varphi_i\ {\rm d}\pmb{v}, \quad 
                             \polC_{ij}^{{\pmb{v}},l} = \int_{\Omega_{\pmb{v}}} v_l\varphi_j\varphi_i\ {\rm d}\pmb{v}, \\
  \polG_{ij}^{\pmb{v},l} &= \mred{\int_{\Omega_{\pmb{v}}} \big(\nabla_{\pmb{v}}\varphi_j \CROSS \pmb{v} \big)_l \varphi_i {\rm d} \pmb{v}} ,
                           \quad l = 1,2,3, \\
  \polM_{ij}^{{\pmb{v}}} &=\int_{\Omega_{\pmb{v}}}\varphi_j\varphi_i\ {\rm d}\pmb{v}, \quad i,j=1,\ldots,N_v.
\end{align*}

\section{A positivity-preserving artificial viscosity scheme}\label{section:pp}
Consider an open bounded domain $\Omega\subset\polR^d$, where $d {=d_x+d_v}>0$ is the dimension. Let us denote $\nabla:=(\nabla_{\bx},\nabla_{\bv})^\mathsf{T}$ and $\pmb{\bbetaa} := (\pmb{v},\, \pmb{E} + \pmb{v} \times \pmb{B})^{\mathsf{T}}$. Let $T>0$ be the final time, and we seek the solution $f$ of the following initial value problems with appropriate boundary conditions in $\Omega\times[0,T]$
\begin{equation}\label{eq:CL}
  \left\{
    \begin{aligned}
      &\partial_t f+\pmb{\bbetaa}\cdot\nabla f=0,\\
      &f(0)=f_0,
    \end{aligned}
  \right. 
\end{equation}
where \mred{$\DIV \pmb{\bbetaa}\equiv0$ because $\nabla_{\pmb{x}}\cdot\pmb{v}=0$ and $\nabla_{\pmb{v}}\cdot(\pmb{E}+\pmb{v}\times\pmb{B})=0$}, and $f_0$ is given initial data.
\mred{Throughout this section, periodic boundary conditions are imposed on the phase-space domain $\Omega$.}

Let $\pmb{\bbetaa}_h$ be the finite element approximation of $\pmb{\bbetaa}$. Then, the Galerkin formulation \eqref{eq:GFEM:ts} is equivalent to the following system
\begin{equation}\label{eq:semi_Vlasov}
  \sum_{j\in\calI(i)}m_{ij}\dot{f_j}
  + 
  \sum_{j\in\calI(i)} c_{ij}(\pmb{\bbetaa}_h) f_j= 0, \qquad i=1,\dots,N,
\end{equation}
where $m_{ij} := (\psi_j, \psi_i)$ is the mass matrix and $c_{ij}(\pmb{\bbetaa}_h) := (\pmb{\bbetaa}_h \cdot \nabla \psi_j, \psi_i)$ is the advection matrix. 

In this section, we introduce a positivity-preserving scheme for \eqref{eq:CL}. The method combines a low-order artificial viscosity scheme that preserves the maximum principle with a high-order viscosity, which is constructed from the residual of \eqref{eq:CL}. We apply a convex limiting method: whenever the high-order solution violates the maximum principle, the scheme reverts locally to the low-order formulation. As a result, the overall method achieves high-order accuracy while strictly maintaining the maximum principle.

\subsection{Low-order scheme}
Let us denote the current time by $t^n\ge 0$ and the current time-step by $\tau_n = t^{n+1}-t^n$. Let $f_h^n := \sum_{j=1}^N f_j^n \psi_j$ be finite element approximation of $f_h$ at $t^n$. We discretize \eqref{eq:semi_Vlasov} in time using the forward Euler scheme, and compute the low-order approximation of $f_h^{\mathsf{L},n+1} := \sum_{j=1}^N f_j^{\mathsf{L},n+1} \psi_j$ from the following scheme:
\begin{equation}\label{eq:LO}
  m_i\frac{f_i^{ {\mathsf{L}},n+1} - f_i^n}{\tau_n} 
  + 
  \sum_{j\in\calI(i)} c_{ij}^n(\pmb{\bbetaa}_h) f_j^n
  -
  \sum_{j\in\calI(i)} d_{ij}^{\mathsf{L},n} f_j^n = 0, \qquad i=1,\dots,N,
\end{equation}
where $m_i = \sum_{j \in \mathcal{I}(i)} m_{ij}$ is the lumped mass matrix, $c_{ij}^n(\pmb{\bbetaa}_h) := (\pmb{\bbetaa}^n_h \cdot \nabla \psi_j, \psi_i)$, and $d_{ij}^{\mathsf{L},n}$ is the low-order viscosity matrix defined below. It is straightforward to show that $\sum_{j \in \mathcal{I}(i)} c^n_{ij}(\pmb{\bbetaa}_h)= 0$, since $\DIV \pmb{\bbetaa}_h \equiv 0$ even in the discrete case. The coefficients $d_{ij}^{\mathsf{L},n}$ are defined as follows:
\begin{equation}\label{eq:dij_L}
  \left\{
    \begin{aligned}
      &d_{ij}^{\mathsf{L},n}= \max\big(|c_{ij}^n(\pmb{\bbetaa}_h)|, |c_{ji}^n(\pmb{\bbetaa}_h)|\big),&\qquad&j\neq i\\
      &d_{ii}^{\mathsf{L},n}= - \sum_{j \in \mathcal{I}(i),j\not=i} d_{ij}^{\mathsf{L},n},
    \end{aligned}\right.
\end{equation}
and the coefficients \(d_{ij}^{\mathsf{L},n}\) satisfy the following properties:
\begin{equation}
  d_{ij}^{\mathsf{L},n} \ge 0 \ (j \neq i), \qquad
  d_{ij}^{\mathsf{L},n} = d_{ji}^{\mathsf{L},n}, \qquad
  \sum_{j \in \mathcal{I}(i)} d_{ij}^{\mathsf{L},n} = 0.\notag
\end{equation}

In order to make the scheme stable and satisfy the discrete maximum principle, the time step $\tau_n$ is bounded by the following CFL condition:
\begin{equation}\label{eq:CFL}
  \tau_n = {\rm cfl}\min_{i=1, \ldots, N} \, \frac{m_i}{\sum_{i\not=j\in\calI(i)}d_{ij}^{\mathsf{L},n}},
\end{equation}
where ${\rm cfl}>0$ is the CFL number. 
Since $m_i\sim h^d$ and $d_{ij}^{\mathsf{L},n}\sim\Vert\pmb{\bbetaa}_h^n\Vert h^{d-1}$, the above CFL condition yields the usual time-stepping restriction used for advection problems, \ie $\tau_n \sim h / \Vert\pmb{\bbetaa}_h^n\Vert$.

\begin{proposition}[Discrete maximum principle]
  \label{prop:low_order}
  The solution obtained from the low-order scheme \eqref{eq:LO} satisfies
  \begin{equation}
    \min_{j \in \mathcal{I}(i)} f_j^n
    \;\le\;
    f_i^{\mathsf{L},n+1}
    \;\le\;
    \max_{j \in \mathcal{I}(i)} f_j^n,
    \qquad i=1,\dots,N,\notag
  \end{equation}
  provided that the time step satisfies the {\rm CFL} condition \eqref{eq:CFL} and the ${\rm CFL}$ number is within the range: $0<{\rm cfl}\leq\frac12$.
\end{proposition}
\begin{proof}
  Since $\sum_{j \in \mathcal{I}(i)} c^n_{ij}(\pmb{\bbetaa}_h)= 0$ and $d_{ii}^{\mathsf{L},n}=-\sum_{j\in\calI(i),j\not=i} d_{ij}^{\mathsf{L},n}$, we get that 
  \[
    \sum_{j\in\calI(i)}c^n_{ij}(\pmb{\bbetaa}_h)f_j^n
    =
    \sum_{j\in\calI(i)}c^n_{ij}(\pmb{\bbetaa}_h)f_j^n
    -
    \sum_{j\in\calI(i)}c^n_{ij}(\pmb{\bbetaa}_h)f_i^n
    =
    \sum_{j\in\calI(i),j\not=i}(f_j^n-f_i^n)c^n_{ij}\notag,
  \]
  and
  \begin{equation}
    \sum_{j\in\calI(i)}d_{ij}^{\mathsf{L},n}f_j^n=\sum_{j\in\calI(i),j\not=i}d_{ij}^{\mathsf{L},n}f_j^n+d_{ii}^{\mathsf{L},n}f_i^n=\sum_{j\in\calI(i),j\not=i}(f_j^n-f_i^n)d_{ij}^{\mathsf{L},n}.\notag
  \end{equation}
  The low-order scheme \eqref{eq:LO} hence becomes
  \begin{equation}
    m_i\frac{f^{\mathsf{L},n+1}_{i} - f^n_{i}}{\tau_n} 
    + 
    \sum_{j\in\calI(i),j\not=i}(f_j^n-f_i^n)c^n_{ij} 
    -
    \sum_{j\in\calI(i),j\not=i}(f_j^n-f_i^n)d_{ij}^{\mathsf{L},n}=0,\notag
  \end{equation} 
  which implies 
  \begin{equation}
    f_i^{\mathsf{L},n+1}=f_i^n\bigg(1-\frac{\tau_n}{m_i}\sum_{j\in\calI(i),j\not=i}(-c^n_{ij}+d_{ij}^{\mathsf{L},n})\bigg)+\frac{\tau_n}{m_i}\sum_{j\in\calI(i),j\not=i}(-c^n_{ij}+d_{ij}^{\mathsf{L},n})f_j^{n}.\notag
  \end{equation} 
  The above formulation can be rewritten as follows:
  \begin{equation}
    f_i^{\mathsf{L},n+1}=\sum_{j\in\calI(i),j\not=i}f_j^n\theta_j^n,\notag
  \end{equation} 
  where $\theta_i^n=1-\frac{\tau_n}{m_i}\sum_{j\in\calI(i),j\not=i}(-c^n_{ij}+d_{ij}^{\mathsf{L},n})$ and $\theta_j^n=\frac{\tau_n}{m_i}(-c^n_{ij}+d_{ij}^{\mathsf{L},n})$ for $j\neq i$.
  If
  $
  \sum_{j\in\mathcal{I}(i)} \theta_j^n = 1
  $
  and
  $
  \theta_j^n \ge 0 \quad \forall j\in\mathcal{I}(i),
  $
  then $f_i^{\mathsf{L},n+1}$ can be expressed as a convex combination of the values from the previous time level $f_j^{n}$. This immediately implies that the scheme satisfies a discrete maximum principle.

  In fact,
  \begin{equation}
    \sum_{j\in\calI(i)}\theta_j^n=1-\frac{\tau_n}{m_i}\sum_{j\in\calI(i),j\not= i}(-c^n_{ij}+d_{ij}^{\mathsf{L},n})+\frac{\tau_n}{m_i}\sum_{j\in\calI(i),j\not= i}(-c^n_{ij}+d_{ij}^{\mathsf{L},n})=1.\notag
  \end{equation}
  The definition of $d_{ij}^{\mathsf{L},n}$ yields $c^n_{ij}\leq d_{ij}^{\mathsf{L},n}$ and $-c^n_{ij}\leq d_{ij}^{\mathsf{L},n}$ when $j\neq i$, which implies
  \begin{equation}
    \theta_j^n=\frac{\tau_n}{m_i}(-c^n_{ij}+d_{ij}^{\mathsf{L},n})\geq0, \quad j \neq i.\notag
  \end{equation}
  Finally, for $j=i$, the following inequality holds
  \begin{equation}
    \theta_i^n=1-\frac{\tau_n}{m_i}\sum_{j\in\calI(i),j\not=i}(-c^n_{ij}+d_{ij}^{\mathsf{L},n})\geq 1-\frac{\tau_n}{m_i}\sum_{j\in\calI(i),j\not=i}(2d_{ij}^{\mathsf{L},n}),\notag
  \end{equation}
  and $\theta_i^n\geq0$ when the CFL condition \eqref{eq:CFL} is satisfied and $0<{\rm cfl}\leq\frac{1}{2}$.
\end{proof}

\subsection{High-order scheme}

We employ a high-order viscosity constructed from the residual of equation~\eqref{eq:vm}, defined as: 
\[
  (\text{residual at } t^n) \equiv D_t f_h^n + \pmb{\bbetaa}_h^n \cdot \nabla f_h^n, 
\]
where $D_t f_h^n$ is the discretisation of the time derivative. At time $t^n$, the finite element residual $R_h(f_h^n)$ is typically constructed using the standard $L^2$-projection, see for instance \cite{MR4907486, MR4456197}: find $R_h(f_h^n) \in \mathcal{V}_h$ such that  
\begin{equation}\label{eq:l2_residual}
  (R_h(f_h^n), \psi_i) = (| D_t f_h^n + \pmb{\bbetaa}_h^n \cdot \nabla f_h^n |, \psi_i),
  \qquad i = 1, \ldots, N, 
\end{equation}
where $D_t f_h^n$ is typically approximated using backward difference formulas (BDF). 

However, the above approach is not applicable when using tensor-product discretizations. One difficulty is that solving the projection problem \eqref{eq:l2_residual} becomes prohibitively expensive in high-dimensional settings. A second issue is that high-order BDF time discretizations \mred{require} storing solutions from multiple previous time levels, leading to increased memory consumption. A third challenge is that the application of nonlinear operations, such as taking the absolute value of the residual, is not straightforward within the tensor product finite element formulation.

The residual viscosity method introduced in \cite{MR3998292} uses residuals of lower dimensional equations derived from the Vlasov equation. While this approach is compatible with tensor-product discretizations and can be seen as a remedy to the projection problem \eqref{eq:l2_residual}, it may be overly diffusive, since artificial diffusion can accumulate in both the physical and velocity dimensions. Moreover, the amount of diffusion depends on the relative scaling of these dimensions, requiring careful parameter tuning to reduce excessive viscosity. In the present work, we propose a new approach for computing the residual of the full Vlasov equation within the tensor-product framework, aiming to retain robustness while significantly reducing artificial diffusion. 

\subsubsection{A mass-lumped residual approximation}
\label{sec:res_mass_lumped}

We observe that a BDF approximation of $D_t f_h^n$ introduces a truncation error associated with the time discretization. Since the governing equation implies the relation
$
\partial_t f = - \pmb{\bbetaa}\cdot\nabla f,
$
we propose to replace the discrete time derivative by a spatial approximation obtained through a \textit{mass-lumped projection}. Specifically, we set $D_t f_h^n \approx \widetilde{R}_{1,h}^n$, where we seek $\widetilde{R}_{1,h}^n \in \widetilde{\calV}_h$ such that
\begin{equation}\label{eq:dtf_projection}
  \widetilde{R}_{1,i}^n
  =
  \frac{1}{\widetilde m_i}
  \bigl( \pmb{ {\bbetaa}}_h^n \cdot \nabla f_h^n, \widetilde{\psi}_i \bigr),
  \qquad i = 1,\ldots,\widetilde N,
\end{equation}
where $\widetilde{R}_{1,i}^n$ is the nodal values of $\widetilde{R}_{1,h}^n$, $\{\widetilde{\psi}_i\}_{i=1}^{\widetilde N}$ denotes the basis of $\widetilde{\mathcal V}_h$, $\widetilde m_i$ are the corresponding lumped mass matrix. The auxiliary space $\widetilde{\mathcal V}_h$ is generally different from $\mathcal V_h$ and may be chosen so as to preserve the formal order of accuracy of the BDF truncation error.

The mass-lumped projection problem introduced above is computationally inexpensive. Using a similar technique, we define a second part of the residual by seeking $R_{2,h}^n \in \calV_h$ such that
\begin{equation}\label{eq:flux_projection}
  R_{2,i}^n
  =
  \frac{1}{m_i}
  \big( \pmb{\bbetaa}_h^n \cdot \nabla f_h^n,\, \psi_i \big),
  \qquad i = 1,\ldots,N
\end{equation}

Finally, we define the residual of the Vlasov equation at time \( t_n \) as
\begin{equation}
  \label{eq:final_residual}
  R_h^n
  :=
  I_h\big(\widetilde{R}_{1,h}^n \big)
  -
  R_{2,h}^n,
\end{equation}
where $I_h : \widetilde{\calV}_h \to \calV_h$ is a linear interpolation operator. 

This construction avoids the evaluation of discrete time derivatives and relies solely on spatial operators, making it particularly suitable for tensor-product discretizations in high-dimensional settings. In this work we construct the space $\widetilde{\mathcal V}_h$ using the same polynomial space as ${\calV}_h$ but on a coarser mesh.

\subsubsection{Residual-viscosity method}

In the previous section, we introduced an efficient approach for approximating the residual of the Vlasov equation. Using this residual approximation, we now construct a high-order residual-based viscosity. To this end, we define the viscosity matrix as follows:
\begin{equation*}
  d_{ij}^{\mathsf{R},n} := m_{ij}\frac{\max(|R_i^n|,|R_j^n|)}{n(f_h^n)}
\end{equation*}
where $n(f_h^n)=\max\big(\Vert f_h^n-{\rm mean}(f_h^n)\Vert_{L^\infty({\Omega})}, 10^{-14} \Vert f_h^n\Vert_{L^\infty({\Omega})} \big)$ is the normalization function. The high-order viscosity coefficient are defined as
\begin{equation}\label{eq:dij_H}
  d_{ij}^{\mathsf{H},n} := \min(d_{ij}^{\mathsf{L},n}, d_{ij}^{\mathsf{R},n}),\qquad d_{ii}^{\mathsf{H},n}=-\sum_{j\in\calI(i),j\not=i}d_{ij}^{\mathsf{H},n},\notag
\end{equation}
and the following properties also hold for $d_{ij}^{\mathsf{H},n}$:
\begin{equation}
  d_{ij}^{\mathsf{H},n}\geq0\ (j\neq i),\qquad d_{ij}^{\mathsf{H},n}= d_{ji}^{\mathsf{H},n},\qquad{\rm and}\ \sum_{j\in\calI(i)}d_{ij}^{\mathsf{H},n}=0.\notag
\end{equation}
The finite element scheme with high-order viscosity is given by
\begin{equation}\label{eq:HO}
  \sum _{j\in\calI(i)} m_{ij}f_j^{\mathsf{H},n+1}=\sum _{j\in\calI(i)} m_{ij}f_j^{n}-\tau_n\sum _{j\in\calI(i)}c^n_{ij}f_j^n+\tau_n\sum _{j\in\calI(i)} d_{ij}^{\mathsf{H},n}f_j^n,
\end{equation}
where the time-step $\tau_n$ is constructed via the CFL condition \eqref{eq:CFL}.

\subsection{Convex limiting}
Let the solution at time $t^n$ be given as $f^n$. Following the discussion in \cite{Guermond_etal_2018}, by subtracting the low-order scheme \eqref{eq:LO} from the high-order scheme \eqref{eq:HO}, we obtain the following relation:  
\begin{equation}
  \sum _{j\in\calI(i)}m_{ij}f_j^{\mathsf{H},n+1}=m_i f_i^{\mathsf{L},n+1}+\sum_{j\in\calI(i)}m_{ij}(f_j^n-f_i^n)+\tau_n\sum_{j\in\calI(i)}(d_{ij}^{\mathsf{H},n}-d_{ij}^{\mathsf{L},n})f_j^n.\notag
\end{equation}
We can rewrite the above equality as follows
\begin{equation}
  \begin{aligned}
    m_if_i^{\mathsf{H},n+1}
    &=m_i f_i^{\mathsf{L},n+1}+\sum_{j\in\calI(i)}m_{ij}(f_j^n-f_i^n)-\sum_{j\in\calI(i)}m_{ij}(f_j^{\mathsf{H},n+1}-f_i^{\mathsf{H},n+1})\\
    &+\tau_n\sum_{j\in\calI(i)}(d_{ij}^{\mathsf{H},n}-d_{ij}^{\mathsf{L},n})(f_j^n-f_i^n) \\
    &=m_if_i^{\mathsf{L},n+1}+\sum_{j\in\calI(i)} A_{ij}^{n+1},\notag
  \end{aligned}
\end{equation}
and the coefficient \( A_{ij}^{n+1}\) is defined as
\begin{equation}\label{eq:co_Aij}
  A_{ij}^{n+1} = 
  m_{ij} \big( (f_i^{\mathsf{H},n+1} - f_i^n) - (f_j^{\mathsf{H},n+1} - f_j^n) \big)
  + \tau_n \big( d_{ij}^{\mathsf{H},n} - d_{ij}^{\mathsf{L},n} \big) (f_j^n - f_i^n).
\end{equation}
It is trivial to show that \( A_{ii}^{n+1} = 0\) and \( A_{ij}^{n+1} = -  A_{ji}^{n+1}\).

We define the limiting matrix $\ell_{ij} = \ell_{ji}$, $\ell_{ij}\in\left[0,1\right]$, and let the solution be given by
\begin{equation}\label{eq:cl_solution}
  m_if_i^{n+1}=m_if_i^{\mathsf{L},n+1}+\sum_{j\in\calI(i)}\ell_{ij} A_{ij}^{n+1}. 
\end{equation}
Let \mred{$\lambda_j(i):=\frac{1}{{\rm card}(\calI(i))-1}$ for $i = 1, \ldots, N$, and $j\in \calI(i) \setminus \{i\}$}, where ${\rm card}(\calI(i))$ is the amount of nodes within the support of $\psi_i$. Then \eqref{eq:cl_solution} becomes 
\begin{equation}\label{eq:conv_comb}
  f_i^{n+1}=\sum_{j\in\calI(i),j\not=i}\lambda_i(f_i^{\mathsf{L},n+1}+\ell_{ij} P_{ij}\mred{^{n+1}}),  
\end{equation}
where \mred{$ P_{ij}^{n+1}:=\frac{1}{\lambda_j(i) m_i} A_{ij}^{n+1}$}. Now take the local maximum and minimal values of $f_h$:
\begin{equation}\label{eq:minmax}
  f^{n,\max}_i = \max_{j\in\calI(i)} f^n_j,\qquad f^{n,\min}_i = \min_{j\in\calI(i)} f^n_j,
\end{equation}
and compute $\ell_j^i$ as follows
\begin{equation}\label{eq:lj}
  \ell_j^i=\left\{
    \begin{aligned}
      &\min\left(\frac{\vert  f^{n,\min}_i-f_i^{\mathsf{L},n+1}\vert}{\vert P_{ij}^{\mred{n+1}}\vert+\epsilon_i},1\right),&\quad &{\rm if}\ f_i^{\mathsf{L},n+1}+ P_{ij}\mred{^{n+1}}< f^{n,\min}_i,\\
      &1,&\quad &{\rm if}\ f^{n,\min}_i\leq f_i^{\mathsf{L},n+1}+ P_{ij}\mred{^{n+1}}\leq  f^{n,\max}_i,  \\
      &\min\left(\frac{\vert  f^{n,\max}_i-f_i^{\mathsf{L},n+1}\vert}{\vert P_{ij}\mred{^{n+1}}\vert+\epsilon_i},1\right),&\quad &{\rm if}\ f^{n,\max}_i<f_i^{\mathsf{L},n+1}+ P_{ij}\mred{^{n+1}},  
    \end{aligned}
  \right.
\end{equation}
where $\epsilon_i = 10^{-14}\max_{i=1,\ldots,N}|f^{n}_i|$ is used to avoid division by zero. Finally, we define
\begin{equation}\label{eq:co_lij}
  \ell_{ij}:=\min(\ell_j^i,\ell_i^j).
\end{equation}

We have the following result for the convex limiting method:
\begin{theorem}
  The convex limiting ~\eqref{eq:cl_solution} conserves the total mass of the system, i.e., $\sum_{i=1}^N m_i f_i^{n+1}=\sum_{i=1}^N m_i f_i^n$. Moreover, the 
  method preserves the discrete maximum principle
  \begin{equation*}
    \min_{j \in \mathcal{I}(i)} f_j^n
    \;\le\;
    f_i^{n+1}
    \;\le\;
    \max_{j \in \mathcal{I}(i)} f_j^n,
    \qquad i = 1,\dots,N.
  \end{equation*}
\end{theorem}

\begin{proof}
  We first prove that the convex limiting is conservative. Summing~\eqref{eq:cl_solution} for $i=1,\ldots,N$, we have:
  \begin{equation}
    \sum_{i=1}^N m_if_i^{n+1}=\sum_{i=1}^N m_if_i^{\mathsf{L},n+1}+\sum_{i=1}^N\sum_{j\in\calI(i)}\ell_{ij} A_{ij}\mred{^{n+1}},\notag
  \end{equation}
  Since $\sum_{j\in\calI(i)}c^n_{ij}=0$ and $\sum_{j\in\calI(i)}d_{ij}^{\mathsf{L},n}=0$, $\sum_{i=1}^N m_if_i^{\mathsf{L},n+1}=\sum_{i=1}^N m_if_i^n$. We also \mred{note} that $\ell_{ij}=\ell_{ji}$, $ A_{ij}\mred{^{n+1}}=- A_{ji}\mred{^{n+1}}$, for $j\neq i$, and $ A_{ii}\mred{^{n+1}}=0$, so $\sum_{i=1}^N\sum_{j\in\calI(i)}\ell_{ij} A_{ij}\mred{^{n+1}}=0$.

  We now prove that the convex limiting algorithm preserves the maximum principle by examining the contribution of each limited correction term. This theorem was introduced in \cite{Guermond_etal_2018} and we provide a proof here for completeness. We consider three cases.

  \medskip
  \noindent\textbf{Step~1.}
  Assume that
  \[
    f_i^{n,\min} \le f_i^{\mathsf{L},n+1} + P_{ij}\mred{^{n+1}} \le f_i^{n,\max}.
  \]
  Since the low-order scheme satisfies the discrete maximum principle, we have
  \[
    f_i^{n,\min} \le f_i^{\mathsf{L},n+1} \le f_i^{n,\max}.
  \]
  Then, for any \( \ell \in [0,\ell_{ij}] \), we can write
  \[
    f_i^{\mathsf{L},n+1} + \ell P_{ij}\mred{^{n+1}}
    =
    (1-\ell) f_i^{\mathsf{L},n+1}
    +
    \ell \bigl( f_i^{\mathsf{L},n+1} + P_{ij}\mred{^{n+1}} \bigr),
  \]
  which is a convex combination of admissible values. Consequently,
  \[
    f_i^{n,\min}
    \le
    f_i^{\mathsf{L},n+1} + \ell_{ij} P_{ij}\mred{^{n+1}}
    \le
    f_i^{n,\max}.
  \]

  \medskip
  \noindent\textbf{Step~2.}
  Assume now that
  \[
    f_i^{\mathsf{L},n+1} + P_{ij}\mred{^{n+1}} < f_i^{n,\min}.
  \]
  Since \( f_i^{\mathsf{L},n+1} \ge f_i^{n,\min} \), this implies \( P_{ij} < 0 \).
  By the definition of the convex limiter~\eqref{eq:lj}, we have
  \begin{equation}\label{eq:step2}
    \begin{aligned}
      f_i^{\mathsf{L},n+1} + \ell_j^i P_{ij}\mred{^{n+1}}
      &=
        f_i^{\mathsf{L},n+1}
        +
        \frac{f_i^{\mathsf{L},n+1} - f_i^{n,\min}}{-P_{ij}\mred{^{n+1}}}\, P_{ij}\mred{^{n+1}}
      \\
      &=
        f_i^{n,\min}.
    \end{aligned}
  \end{equation}

  \medskip
  \noindent\textbf{Step~3.}
  The remaining case,
  \[
    f_i^{\mathsf{L},n+1} + P_{ij}\mred{^{n+1}} > f_i^{n,\max},
  \]
  $P_{ij}>0$ since $f_i^{\mathsf{L},n+1}\le f_i^{n,\max}$. Thus,
  \begin{equation}\label{eq:step3}
    \begin{aligned}
      f_i^{\mathsf{L},n+1} + \ell_j^i P_{ij}\mred{^{n+1}}
      &=
        f_i^{\mathsf{L},n+1}
        +
        \frac{f^{n,\max}_i-f_i^{\mathsf{L},n+1}}{P_{ij}\mred{^{n+1}}}\, P_{ij}\mred{^{n+1}}
      \\
      &=
        f_i^{n,\max}.
    \end{aligned}
  \end{equation}
	
Finally, Steps 2 and 3 give that either $f_i^{\mathsf{L},n+1} + \ell_j^i P_{ij}\mred{^{n+1}}=f_i^{n, \min}$ or $f_i^{\mathsf{L},n+1} + \ell_j^i P_{ij}\mred{^{n+1}}=f_i^{n, \max}$. Defining \( \ell := \ell_{ij}/\ell_j^i \in [0,1] \), we obtain
  \[
    f_i^{\mathsf{L},n+1} + \ell_{ij} P_{ij}\mred{^{n+1}}
    =
    (1-\ell) f_i^{\mathsf{L},n+1}
    +
    \ell \bigl( f_i^{\mathsf{L},n+1} + \ell_j^i P_{ij}\mred{^{n+1}} \bigr),
  \]
  which is again a convex combination of admissible values. Hence,
  \[
    f_i^{n,\min}
    \le
    f_i^{\mathsf{L},n+1} + \ell_{ij} P_{ij}\mred{^{n+1}}
    \le
    f_i^{n,\max}.
  \]

  \medskip
  Finally, equation~\eqref{eq:conv_comb} shows that \( f_i^{n+1} \) is a convex combination of terms that individually satisfy the discrete maximum principle. Therefore,
  \[
    f_i^{n,\min} \le f_i^{n+1} \le f_i^{n,\max},
    \qquad i = 1,\ldots,N,
  \]
  which completes the proof.
\end{proof}

\begin{remark}[Relaxation of the local extrema]
  Strictly enforcing the maximum principle removes oscillations in the residual viscosity discretization but degrades the accuracy of the scheme; see the discussion in \cite[Sec. 4.7]{Guermond_etal_2018} and references therein. To retain optimal accuracy, we followed the above reference and relaxed the local maximum and minimum values in \eqref{eq:minmax} as follows:
 
   {First, we set $\Delta^2 f_i^n:=\sum_{j\in\calI(i),j\not=i}f_i^n-f_j^n$, $i=1,\ldots,N$, and then define
  \begin{equation}
  \overline{\Delta^2 f_i^n}:=\frac{1}{2{\rm card}(\calI(i))}\sum_{j\in\calI(i),j\not=i}\left(\frac{1}{2}\Delta^2f_i^n+\frac{1}{2}\Delta^2f_j^n\right).\notag
  \end{equation}
  Instead of using the local maximum $f_i^{n,\max}$ and minimal $f_i^{n,\min}$ as bounds for the limiter, we choose 
  \begin{equation}\label{eq:relaxation}
    \begin{aligned}
      \overline{f_i^{n,\max}}=\min((1+r_h)f_i^{n,\max},f_i^{n,\max}+\vert\overline{\Delta^2 f_i^n}\vert),\\
      \overline{f_i^{n,\min}}=\max((1-r_h)f_i^{n,\min},f_i^{n,\min}-\vert\overline{\Delta^2 f_i^n}\vert),
    \end{aligned}
  \end{equation}}
where $r_h = \big(\frac{m_i}{\vert\Omega\vert}\big)^{\frac{1.5}{d}}$ is \mred{a mesh-dependent relaxation parameter satisfying $r_h\sim h^{1.5}$ on quasi-uniform meshes.} We note that without this relaxation, we obtained optimal convergence rates in all norms presented in this paper, except in the \(L^\infty\)-norm. \mred{The relaxed bounds no longer imply the strict local discrete maximum principle stated above. Positivity nevertheless survives when $0\le r_h\le1$, since $\overline{f_i^{n,\min}}\ge (1-r_h)f_i^{n,\min}\ge0$. Thus, all production runs using the relaxed limiter remain positivity preserving.}

\end{remark}

\section{Finite element methods for Maxwell's equations}\label{section:FEMMaxwell}
\subsection{Semi-discretization of the Maxwell's equations}
Let $\rho_h$ and $\pmb{J}_h$ be the approximations of charge and current densities, respectively. We compute them as follows
\begin{equation}\label{eq:densities}
  \rho_h = \int_{\Omega_{\pmb{v}}}f_h\ {\rm d}\pmb{v}-\overline{\rho_h},\qquad 
  \pmb{J}_h = \int_{\Omega_{\pmb{v}}}\pmb{v}f_h\ {\rm d}\pmb{v},
\end{equation}
$\overline{\rho_h}=\int_{\Omega_{\pmb{x}}}\rho_h\ {\rm d}\pmb{x}/|\Omega_{\pmb{x}}|$. Since \mred{$f_h\in\calV_h:=\calV_{\pmb{x}}\otimes\calV_{\pmb{v}}$}, it follows that $\rho_h\in\calV_{\pmb{x}}$ and $\pmb{J}_h\in[\calV_{\pmb{x}}]^{d_{ {x}}}$. Let \(\Phi_h \in \mathcal{V}_{\pmb{x}}\) be the finite element approximation of the electric potential, obtained by solving the corresponding Poisson equation. We refer the reader to \cite{MR4907486, MR4945433} for details on the solution of the Poisson equation and the computation of \(\rho_h\) and \(\pmb{J}_h\).

The spaces of electromagnetics form a de Rham complex. To define appropriate function spaces for the discrete electromagnetic fields, we follow the framework of finite element exterior calculus (FEEC), according to the structure of the de~Rham complex:
\begin{equation}
  \begin{tikzcd}
    H^1(\Omega_{\pmb{x}}) \arrow[r, "\text{grad}"] \arrow[d, "\Pi_0"] &H(\text{curl}; \Omega_{\pmb{x}})\arrow[r, "\text{curl}"] \arrow[d, "\Pi_1"] & H(\text{div}; \Omega_{\pmb{x}}) \arrow[r, "\text{div}"] \arrow[d, "\Pi_2"] & L^2(\Omega_{\pmb{x}}) \arrow[d, "\Pi_3"] \\
    V_0 \arrow[r, "\text{grad}"] & V_1 \arrow[r, "\text{curl}"] & V_2 \arrow[r, "\text{div}"] & V_3
  \end{tikzcd}\notag
\end{equation}
and in the diagram above, \(V_0\) corresponds to the space \(\mathcal{V}_{\pmb{x}}\) we defined before. 
We further define the remaining spaces using first-order N\'ed\'elec and Raviart--Thomas elements:
\[
  \calN_{1}(\widehat{K}) 
  := 
  \big[ \polQ_{0, 1, 1}(\widehat{K}), \polQ_{1, 0, 1}(\widehat{K}), \polQ_{1, 1, 0}(\widehat{K}) \big],
\]
and 
\[
  \calR\calT_{1}(\widehat{K}) 
  := 
  \big[ 
  \polQ_{1, 0, 0}(\widehat{K}), 
  \polQ_{0, 1, 0}(\widehat{K}), 
  \polQ_{0, 0, 1}(\widehat{K}) \big],
\]
where $\polQ_{\alpha_1, \alpha_2, \alpha_3}$ composed of polynomials whose degree with respect to $x_i$ is at most $\alpha_i$. Then, we define the following function spaces:
\[
  \pmb{\polE}_h:=
  \{
  \bw\in H(\textrm{curl};\Omega_{{\pmb{x}} }): [\GRAD_{\widehat{{\pmb{x}} }}\bT_K(\widehat{{\pmb{x}} }))]^\top \bw(\bT_K({\widehat{{\pmb{x}} }})) \in \calN_1(\widehat{K}), \forall K\in \calT_{{\pmb{x}} }
  \},
\]
and
\[
  \begin{aligned}
    \pmb{\polB}_h:=
    \{
    \bw & \in H(\textrm{div};\Omega_{{\pmb{x}} }): \\
        &
          \det(\GRAD_{\widehat{{\pmb{x}} }}\bT_K(\widehat{{\pmb{x}} }))
          [\GRAD_{\widehat{{\pmb{x}} }}\bT_K(\widehat{{\pmb{x}} }))]^{-1} 
          \bw(\bT_K({\widehat{{\pmb{x}} }})) \in \calR\calT_1(\widehat{K}), 
          \forall K\in \calT_{{\pmb{x}} }
          \}.
  \end{aligned}
\]
Within the de~Rham complex, $\pmb{\polE}_h$ and $\pmb{\polB}_h$ correspond to the spaces \(V_1\) and \(V_2\), respectively.

With these spaces at hand, we obtain the semi-discretization of the Maxwell's equations as follows: for given $\pmb{J}_h\in\calC^1([0,T];[\calV_{\pmb{x}}]^{d_{ {x}}})$ find $\pmb{E}_h \in \calC^1([0,T];\ \pmb{\polE}_h)$ and $\pmb{B}_h\in \calC^1([0,T];\ \pmb{\polB}_h)$ such that 
\begin{equation}\label{eq:fem:AP11}
  \begin{aligned}
    (\partial_t \pmb{E}_h,\pmb{\eta} )
    &=(\pmb{B}_h,\nabla_{\pmb{x}}\times\pmb{\eta})-(\pmb{J}_h,\pmb{\eta}), \qquad \forall \pmb{\eta} \in \pmb{\polE}_h,
    \\
    (\partial_t\pmb{B}_h, \pmb{\xi})&=-(\nabla_{{\pmb{x}} }\CROSS\pmb{E}_h, \pmb{\xi}), \qquad \forall \pmb{\xi} \in \pmb{\polB}_h.
  \end{aligned}
\end{equation}

\begin{proposition} The divergence constraint of $\pmb{B}_h$ is satisfied, i.e.,
  \begin{equation}
    \nabla_{\pmb{x}}\cdot\pmb{B}_h=0,\qquad a.e.\qquad\forall t>0,
  \end{equation}
  if it holds when $t=0$.
\end{proposition}
\begin{proof}
  Since $\pmb{E}_h\in\pmb{\polE}_h$, $\nabla_{\pmb{x}}\times\pmb{E}_h\in\pmb{\polB}_h$, by the construction of the discrete de Rham structure. Hence, $\partial_t(\nabla_{\pmb{x}}\cdot\pmb{B}_h)=-\nabla_{\pmb{x}}\cdot\nabla_{\pmb{x}}\times\pmb{E}_h=0$ almost everywhere. 
\end{proof}

\subsection{A divergence cleaning technique}
By integrating the Vlasov equation over the velocity space in the \mred{continuous} setting, we obtain the continuity equation for charge conservation:
\begin{equation}\label{eq:continuity_eq}
  \partial_t\rho + \nabla_\bx\cdot\pmb{J}=0.
\end{equation}
This equation is crucial for preserving Gauss's law for the electric field, i.e., $\nabla_{\pmb{x}}\cdot\pmb{E}=\rho$. In the standard finite element discretization of the Vlasov equation, the variational formulation is employed, and the  Maxwell's equation $\partial_t\pmb{E}_h=\nabla_{\pmb{x}}\times\pmb{B}_h-\pmb{J}_h$ is solved in the weak form. \mred{Under periodic boundary conditions (or boundary conditions for which the boundary term vanishes), integration by parts gives $(\nabla_{\pmb{x}}\cdot\pmb{E}_h,\phi_i)=-(\pmb{E}_h,\nabla_{\pmb{x}}\phi_i)$.} As a consequence, Gauss's law is preserved in the weak sense:
\begin{equation}\label{eq:weak_gauss}
  (\pmb{E}_h,\nabla_{\pmb{x}}\phi_i) = \mred{-(\rho_h,\phi_i)},\qquad i=1,\ldots,N_x.
\end{equation}

However, the introduction of artificial viscosity modifies the Vlasov equation, so that the continuity equation~\eqref{eq:continuity_eq} is no longer satisfied exactly and, consequently, Gauss’s law may be violated. In our previous work~\cite{MR4945433}, this issue was addressed by introducing a correction term derived from the artificial viscosity. In the present work, however, the method involves several artificial diffusion mechanisms and nonlinear limiting steps, making a direct extension of the correction strategy proposed in~\cite{MR4945433} nontrivial. 

Instead, we employ a divergence-cleaning technique introduced in~\cite{munz:2000}, in which the cleaning procedure is applied after each time step. Below, we present the technique.

At each time step, using forward Euler, we update $\pmb{E}_h$ and $\pmb{B}_h$ as follows:
\begin{equation}\label{eq:Euler_Maxwell}
  \begin{aligned}
    \Big(\frac{\widehat{\pmb{E}_h^{n+1}}-\pmb{E}_h^n}{\tau_n},\pmb{\eta}\Big)
    &=(\pmb{B}^n_h,\nabla_{\pmb{x}}\times\pmb{\eta})-(\pmb{J}^n_h,\pmb{\eta}), \qquad \forall \pmb{\eta} \in \pmb{\polE}_h,
    \\
    \Big(
    \frac{\pmb{B}_h^{n+1}-\pmb{B}_h^n}{\tau_n},
    \pmb{\xi}
    \Big)
    &=
      -(\nabla_{{\pmb{x}} }\CROSS\pmb{E}_h^n, \pmb{\xi}),
      \qquad \pmb{\xi} \in \pmb{\polB}_h.
  \end{aligned}
\end{equation}
After obtaining $\widehat{\pmb{E}_h^{n+1}}$, we correct it by
\begin{equation}\label{eq:E_correction}
  \pmb{E}_h^{n+1} = \widehat{\pmb{E}_h^{n+1}}-\delta \pmb{E}_h^{n+1},
\end{equation}
where $\delta \pmb{E}_h^{n+1}\in\pmb{\polE}_h$ and 
\begin{equation}\label{eq:dEn}
  \delta\pmb{E}_h^{n+1} = -\nabla_\bx(\delta\Phi_h^{n+1}),
\end{equation}
and the term $\delta\Phi_h^{n+1}\in\calV_{\pmb{x}}$ is obtained by solving the following variational problem for the Poisson equation: find $\delta\Phi_h^{n+1}\in\calV_{\pmb{x}}$ such that
\begin{equation}\label{eq:dPhin}
  (\nabla_\bx(\delta\Phi_h^{n+1}),\nabla_\bx\phi_i) = -(\widehat{\pmb{E}_h^{n+1}},\nabla_\bx\phi_i) \mred{- (\rho_h^{n+1},\phi_i)},\qquad i=1,\ldots,N_x.
\end{equation}

\begin{proposition}
  After the divergence cleaning, the electric field $\pmb{E}_h^n$ satisfy the Gauss's law in the weak form:
  
  \begin{equation}
    (\pmb{E}_h^n,\nabla_{\pmb{x}}\phi_i) = \mred{-(\rho_h^n,\phi_i)},\qquad i=1,\ldots,N_x.\notag
  \end{equation}
\end{proposition}
\begin{proof}For any function $\phi_i\in\calV_{\pmb{x}}$, we can deduce:
  \begin{equation}
    \begin{aligned}
      (\pmb{E}_h^{n+1},\nabla_{\pmb{x}}\phi_i) &= (\widehat{\pmb{E}_h^{n+1}}-\delta \pmb{E}_h^{n+1},\nabla_{\pmb{x}}\phi_i)= (\widehat{\pmb{E}_h^{n+1}},\nabla_{\pmb{x}}\phi_i)-(\delta \pmb{E}_h^{n+1},\nabla_{\pmb{x}}\phi_i) \\
                                               &=(\widehat{\pmb{E}_h^{n+1}},\nabla_{\pmb{x}}\phi_i)+(\nabla_\bx(\delta\Phi_h^{n+1}),\nabla_{\pmb{x}}\phi_i)\\
                                               &=(\widehat{\pmb{E}_h^{n+1}},\nabla_{\pmb{x}}\phi_i)-(\widehat{\pmb{E}_h^{n+1}},\nabla_{\pmb{x}}\phi_i)\mred{- (\rho_h^{n+1},\phi_i)}= \mred{-(\rho_h^{n+1},\phi_i)}.
    \end{aligned}\notag
  \end{equation}
\end{proof}

\subsection{Time discretization and summary of the algorithms}
In Sec.~\ref{section:pp}, we introduced the low-order and high-order viscosity schemes together with the convex limiter. For temporal discretization, we employ high-order strong-stability-preserving Runge–Kutta (SSPRK) schemes to achieve high accuracy in time. In this work, we use the third-order SSP Runge--Kutta (SSP-RK3) scheme presented in \cite{Shu_Osher1988}.

The SSP-RK3 scheme can be written as a convex combination of Forward Euler steps and therefore inherits a positivity property the Forward Euler scheme under the same CFL condition. For the semi-discrete equation $\partial_t u = g(u)$, the update from $u^n$ to $u^{n+1}$ at each time step is given by
\begin{equation}\label{eq:rk3}
  \begin{aligned}
    u^{(0)} &= u^n, \\
    u^{(1)} &= u^{(0)} + \tau_n g(u^{(0)}), \quad &\textit{Forward-Euler},\\
    u^{(2)} &= u^{(1)} + \tau_n g(u^{(1)}), \quad &\textit{Forward-Euler},\\
    \overline{u}^{(2)} &= \tfrac34 u^{(0)} + \tfrac14 u^{(2)}, \quad &\textit{convex combination},\\
    u^{(3)} &= \overline{u}^{(2)} + \tau_n g(\overline{u}^{(2)}), \quad &\textit{Forward-Euler},\\
    u^{n+1} &= \tfrac13 u^{(0)} + \tfrac23 u^{(3)}, \quad &\textit{convex combination}.\\
  \end{aligned}
\end{equation}
\mred{This ordering is algebraically equivalent to the usual representation of SSP-RK3 \cite{Shu_Osher1988}: the first convex combination is formed before evaluating the third Forward--Euler stage.}

\begin{algorithm}[!t]
  \renewcommand{\algorithmicrequire}{\textbf{Input:}}
  \renewcommand{\algorithmicensure}{\textbf{Output:}}
  \caption{Time-stepping procedure for the Vlasov equation}
  \label{alg:Vlasov}
  \begin{algorithmic}[1]
    \Require {$f_h^n$, $\pmb{E}_h^n$ and $\pmb{B}_h^n$.} 
    \Ensure {$f_h^{n+1}$, $\pmb{E}_h^{n+1}$ and $\pmb{B}_h^{n+1}$.}
    \While{$t < T$}

    \State Compute $\pmb{\bbetaa}_h^n := (\pmb{v},\, \pmb{E}_h^n + \pmb{v} \times \pmb{B}_h^n)^{\mathsf{T}}$;
    \State Compute $R_h^n$ using \eqref{eq:final_residual};
    \State Compute the time step $\tau_n$ using the CFL condition \eqref{eq:CFL};
    \State Set $f_h^{(0)} = f_h^{n}$, 
    $\pmb{E}_h^{(0)} = \pmb{E}_h^{n}$ and $\pmb{B}_h^{(0)} = \pmb{B}_h^{n}$;
    
    \For {each Forward-Euler stages of SSP-RK3 \eqref{eq:rk3}: $k=0$ to $2$ }

    \If {$k=2$}
    \State Update
    \begin{equation*}
      f_h^{(k)} = \tfrac34 f_h^{(0)} + \tfrac14 f_h^{(k)}, \qquad
      \pmb{E}_h^{(k)} = \tfrac34 \pmb{E}_h^{(0)} + \tfrac14 \pmb{E}_h^{(k)}, \qquad
      \pmb{B}_h^{(k)} = \tfrac34 \pmb{B}_h^{(0)} + \tfrac14 \pmb{B}_h^{(k)}.
    \end{equation*}
    \EndIf
    
    \State Compute the viscosities $d_{ij}^{\mathsf{L},(k)}$ \eqref{eq:dij_L} and $d_{ij}^{\mathsf{H},(k)}$ \eqref{eq:dij_H};
    \State Compute the low-order solution $f_h^{\mathsf{L},(k+1)}$ \eqref{eq:LO}
    and high-order solution $f_h^{\mathsf{H}, (k+1)}$ \eqref{eq:HO};
    \State Compute the coefficients $ A_{ij}$ \eqref{eq:co_Aij} and $\ell_{ij}$ \eqref{eq:co_lij};
    \State Apply the limiter \eqref{eq:cl_solution}, the relaxation \eqref{eq:relaxation}, and obtain $f_h^{(k+1)}$;

    \State Compute $\pmb{E}_h^{(k+1)}$ and $\pmb{B}_h^{(k+1)}$ from \Call{Algorithm \ref{alg:Maxwell}}{$f^{(k+1)}_h$, $\pmb{E}_h^{(k)}$, $\pmb{B}_h^{(k)}$, $\tau_n$};
    
    \EndFor
    
    \State Compute the convex combination in \eqref{eq:rk3}: 
    \begin{equation*}
      f_h^{n+1} = \tfrac13 f_h^{(0)} + \tfrac23 f_h^{(3)}, \qquad
      \pmb{E}_h^{n+1} = \tfrac13 \pmb{E}_h^{(0)} + \tfrac23 \pmb{E}_h^{(3)}, \qquad
      \pmb{B}_h^{n+1} = \tfrac13 \pmb{B}_h^{(0)} + \tfrac23 \pmb{B}_h^{(3)}.
    \end{equation*}

    \State Update time: $t \leftarrow t + \tau_n$, $n \leftarrow n + 1$ ;
    \EndWhile
  \end{algorithmic}
\end{algorithm}

The Vlasov and Maxwell equations are coupled using the same time integrator for the full Vlasov–Maxwell system. At each time step, the time discretizations of the Vlasov equation is described in Algorithms~\ref{alg:Vlasov}. Note, that at the end of each time levels the Maxwell equation is solved in order to update the electric and magnetic fields. 

\begin{algorithm}[!t]
  \renewcommand{\algorithmicrequire}{\textbf{Input:}}
  \renewcommand{\algorithmicensure}{\textbf{Output:}}
  \caption{Procedure for the Maxwell update}
  \label{alg:Maxwell}
  \begin{algorithmic}[1]
    \Require {$f^{( {k+1})}_h$, $\pmb{E}_h^{( {k})}$, $\pmb{B}_h^{( {k})}$, $\tau_n$}
    \Ensure {$\pmb{E}_h^{({ {k}}+1)}$ and $\pmb{B}_h^{( {k}+1)}$.}
    \State Compute $\pmb{J}_h^{( {k})}$ using \eqref{eq:densities};
    \State Obtain $\widehat{\pmb{E}_h^{( {k}+1)}}$ and $\pmb{B}_h^{( {k}+1)}$ using \eqref{eq:Euler_Maxwell};
    \State Compute $\rho_h^{( {k}+1)}$ using \eqref{eq:densities}; 
    \State Compute $\delta\Phi_h^{( {k}+1)}$ \eqref{eq:dPhin} and $\delta E_h^{( {k}+1)}$ \eqref{eq:dEn};
    \State Perform the divergence cleaning \eqref{eq:E_correction} and obtain $\pmb{E}_h^{( {k}+1)}$;
  \end{algorithmic}
\end{algorithm}

\section{Numerical experiments}\label{section:experiment}
In this section, we present numerical results for our methods applied to several benchmark problems for both the Vlasov--Poisson and Vlasov--Maxwell equations. 

The method presented in this manuscript is implemented in DOLFINx \cite{baratta2023dolfinx}, the problem-solving environment of the FEniCS project. DOLFINx provides a high-level C++/Python interfaces for finite element computations, supporting automated code generation, parallel execution, and adaptive mesh refinement. It allows us to efficiently solve partial differential equations on complex domains while making use of modern parallel computing architectures.

In the residual computation presented in Section~\ref{sec:res_mass_lumped}, the space 
$\tilde{\calV}_h$ requires interpolation between different meshes, which can be performed using the built-in functions in DOLFINx. In the results presented below, the space $\tilde{\calV}_h$ is constructed using the same polynomial space as $\calV_h$ but on coarser meshes. For example, $\tilde{\calV}_h = \calV_{ {qh}}$ denotes that the space $\tilde{\calV}_h$ is constructed on a mesh that is $q$ times coarser then the mesh used for $\calV_h$. Therefore, $d_{ij}^{\mathsf{H},  {qh}}$ corresponds to the residual \eqref{eq:final_residual} computed using $R^n_{1, {qh}} \in \calV_{ {qh}}$. For instance, if $\calV_h$ corresponds to a $100 \times 100$ mesh of $\mathbb{Q}_1$ elements, then $d_{ij}^{\mathsf{H},  {4h}}$ corresponds to the case where the residual is evaluated on the coarser space $\tilde{\calV}_h = \calV_{ {4h}}$ with a $25 \times 25$ mesh of $\mathbb{Q}_1$ elements.

We first present numerical validations for the 1D1V Vlasov--Poisson equation. Then, we show benchmark results for the Vlasov--Maxwell system in 1D2V and 2D2V settings. Unless otherwise specified, the CFL number is set to $0.3$. Additional numerical experiments are reported in the supplementary material.

\subsection{1D1V Vlasov--Poisson equations}
\subsubsection{Landau damping}
First, we test the accuracy of the method through a convergence study. For the Vlasov equations, as introduced in \cite{MR4907486,MR4945433,MR3267101}, let $f(\pmb{x}, \pmb{v}, 0)$ be the initial condition. We then solve the Vlasov--Poisson equation to obtain the solution $f(\pmb{x}, \pmb{v}, T)$ at $t = T$. Now, if we use $f(\pmb{x}, -\pmb{v}, T)$ as the initial data and solve the Vlasov--Poisson equation again, the resulting solution at $t = T$ should be $f(\pmb{x}, -\pmb{v}, 0)$.

Let the phase space be $\Omega:=[0,4\pi ]\times [-6,6]$, and use the following initial data
\begin{equation}
  f(x,v,0)=\frac{1}{\sqrt{2\pi}}{\rm exp}\left(-\frac{v^2}{2}\right)(1+\alpha{\rm cos}(k x)), \notag 
\end{equation}
where $k=0.5$ and $\alpha=0.01$. The aim of this benchmark is to assess the accuracy of the residual viscosity method constructed using several spaces $\tilde{\calV}_h$. The simulation results are reported in Table~\ref{tab:convergence1}. The error norms are slightly larger for coarser spaces $\tilde{\calV}_h$, as expected. However, the convergence rates are optimal in all cases.
 
We then apply the convex limiter together with the relaxation technique and assess the accuracy of the resulting scheme; the results are shown in Table~\ref{tab:convergence2}. We observe that, in terms of the $L^1$- and $L^2$-norms, the errors remain comparable to those obtained without limiting and continue to converge at second order when the limiter is applied. We also observe that, for the mesh resolutions tested here, when the relaxation is not applied, the $L^\infty$-norms of the errors still exhibit second-order convergence, rather than the reduced rate reported in \cite{Guermond_etal_2018}; however, the errors are significantly larger than those obtained without the convex limiter. When the relaxation technique is applied, the errors are reduced to the same level as those obtained without the limiter, particularly when the mesh is sufficiently refined.

\begin{table}[htbp]
  \caption{Linear Landau damping: convergence rates obtained using Galerkin finite elements and the residual viscosity method with different residuals. The reported errors are for $f$ at \mred{$t=10$} in the forward--backward test: the simulation is run to $t=5$, the velocity is reversed, and it is continued to $t=10$. The parameters are \mred{$k=0.5$}, $\alpha = 0.01$, and \mred{${\rm cfl}=0.3$}.
} 
  \label{tab:convergence1}
  \small
  \begin{center}
    \begin{tabular}{c|c|c|c|c|c|c|c} \hline
      &$\#{\rm elements}$&$L^1$-error&Rate&$L^2$-error&Rate&$L^\infty$-error&Rate\\ \hline
      \multirow{5}{*}{Galerkin} &$32\times32$ & 1.12E-02&-&1.11E-02 &- &1.49E-02&- \\
      &$64\times64$&2.83E-03&1.99&2.78E-03&2.00&3.87E-03&1.95 \\
      &$128\times128$&7.08E-04&2.00&6.95E-04&2.00&9.76E-04&1.99 \\
      &$256\times256$&1.77E-04&2.00&1.74E-04&2.00&2.45E-04&2.00  \\
      &$512\times512$&4.43E-05&2.00&4.34E-05&2.00&6.12E-05&2.00  \\ \hline
      \multirow{5}{*}{RV with $d_{ij}^{\mathsf{H}, 2h}$} &$32\times32$ & 1.12E-02&-&1.11E-02 &- &1.52E-02&- \\
      &$64\times64$&2.87E-03&2.03&2.81E-03&2.03&3.90E-03&1.96 \\
      &$128\times128$&7.12E-04&2.10&6.97E-04&2.01&9.79E-04&1.99 \\
      &$256\times256$&1.77E-04&2.00&1.74E-04&2.00&2.45E-04&2.00  \\
      &$512\times512$&4.43E-05&2.00&4.34E-05&2.00&6.12E-05&2.00  \\ \hline
      \multirow{5}{*}{RV with $d_{ij}^{\mathsf{H}, 4h}$} &$32\times32$ & 1.22E-02&-&1.19E-02 &- &1.55E-02&- \\
      &$64\times64$&2.96E-03&2.04&2.88E-03&2.04&3.94E-03&1.97 \\
      &$128\times128$&7.20E-04&2.04&7.04E-04&2.03&9.84E-04&2.00 \\
      &$256\times256$&1.78E-04&2.02&1.74E-04&2.01&2.45E-04&2.00\\
      &$512\times512$&4.44E-05&2.00&4.35E-05&2.00&6.12E-05&2.00  \\ \hline
      \multirow{5}{*}{RV with $d_{ij}^{\mathsf{H}, 8h}$} &$32\times32$ & 1.28E-02&-&1.22E-02 &- &1.58E-02&- \\
      &$64\times64$&3.09E-03&2.04&2.98E-03&2.03&3.99E-03&1.98 \\
      &$128\times128$&7.44E-04&2.05&7.19E-04&2.05&9.91E-04&2.01 \\
      &$256\times256$&1.80E-04&2.05&1.76E-04&2.03&2.47E-04&2.01 \\
      &$512\times512$&4.45E-05&2.02&4.36E-05&2.01&6.13E-05&2.01  \\ \hline
    \end{tabular}
  \end{center}
\end{table}

\begin{table}[htbp]
  \caption{Linear Landau damping: errors of finite element solutions with convex limiter, without and with relaxation technique applied. The reported errors are for $f$ at \mred{$t=10$} in the forward--backward test: the simulation is run to $t=5$, the velocity is reversed, and it is continued to $t=10$. The parameters are \mred{$k=0.5$}, $\alpha = 0.01$, and \mred{${\rm cfl}=0.3$}.}
  \label{tab:convergence2}
  \begin{center}
  \small
    \begin{tabular}{c|c|c|c|c|c|c|c} \hline
      relaxation&$\#{\rm elements}$&$L^1$-error&Rate&$L^2$-error&Rate&$L^\infty$-error&Rate\\ \hline
      \multirow{5}{*}{without}&$32\times32$ & 1.28E-02&-&1.54E-02 &- &3.39E-02&- \\
                &$64\times64$ & 2.92E-03&2.14&3.03E-03&2.34&8.49E-03&2.00 \\
                &$128\times128$&7.15E-04&2.03&7.11E-04&2.09&2.12E-03&2.00 \\
                &$256\times256$&1.78E-04&2.01&1.75E-04&2.02&5.31E-04&2.00  \\
                &$512\times512$&4.43E-05&2.00&4.35E-05&2.01&1.33E-04&2.00  \\  \hline
      \multirow{5}{*}{with}&$32\times32$ & 1.17E-02&&1.14E-03&&1.52E-03& \\
                &$64\times64$ &2.87E-03&2.03&2.81E-03&2.03&3.90E-03&1.96 \\ 
                &$128\times128$&7.12E-04&2.01&6.97E-04&2.01&9.79E-04&1.99\\
                &$256\times256$&1.77E-04&2.00&1.74E-04&2.00&2.45E-04&2.00  \\
                &$512\times512$&4.43E-05&2.00&4.34E-05&2.00&6.12E-05&2.00  \\\hline
    \end{tabular}
  \end{center}
\end{table}
For all remaining tests, unless otherwise specified, we use  $R^n_{1, {2h}} \in \calV_{ {2h}}$ and $d_{ij}^{\mathsf{H},  {2h}}$, and we apply the convex limiting together with the relaxation.

Now, we set $\alpha=0.5$, and look at the strong Landau damping. The solution $f_h$ at $t = 30$ is plotted in Figure~\ref{fig:LD_t30}, where similar structures to those reported in \cite{MR2806222, MR4329985, MR2576241} can be observed. We observe that, when using the Galerkin finite element method, the extrema significantly violate the physical bounds, with the minimum reaching approximately  {$-0.14$}. This behavior may be caused by numerical errors and spurious oscillations. 
When using the residual viscosity schemes, the extrema are improved, as the artificial viscosity helps suppress these oscillations; however, it does not guarantee the positivity of the distribution function. By applying the convex limiter, positivity preserving is achieved. The resulting distribution is similar to that obtained without the convex limiter, but confined to a smaller range. We depict \mred{$\log_{10}$ of the electric energy $\mathcal{E}_1:=\tfrac12\Vert E_1\Vert_{L^2(\Omega_x)}^2$} in Figure~\ref{fig:LD}. We obtain fitted damping and growth rates of $-0.2854$ and $0.0867$, respectively, which \mred{compare well with the standard reference values $-0.2920$ and $0.0815$} reported in \cite[Figure 5]{Kraus_Kormann_Morrison_Sonnendruecker_2017}, regardless of whether the limiter is applied.
\begin{figure}[htbp]
  \centering
  \begin{subfigure}[b]{0.32\textwidth}
    \includegraphics[width=\linewidth]{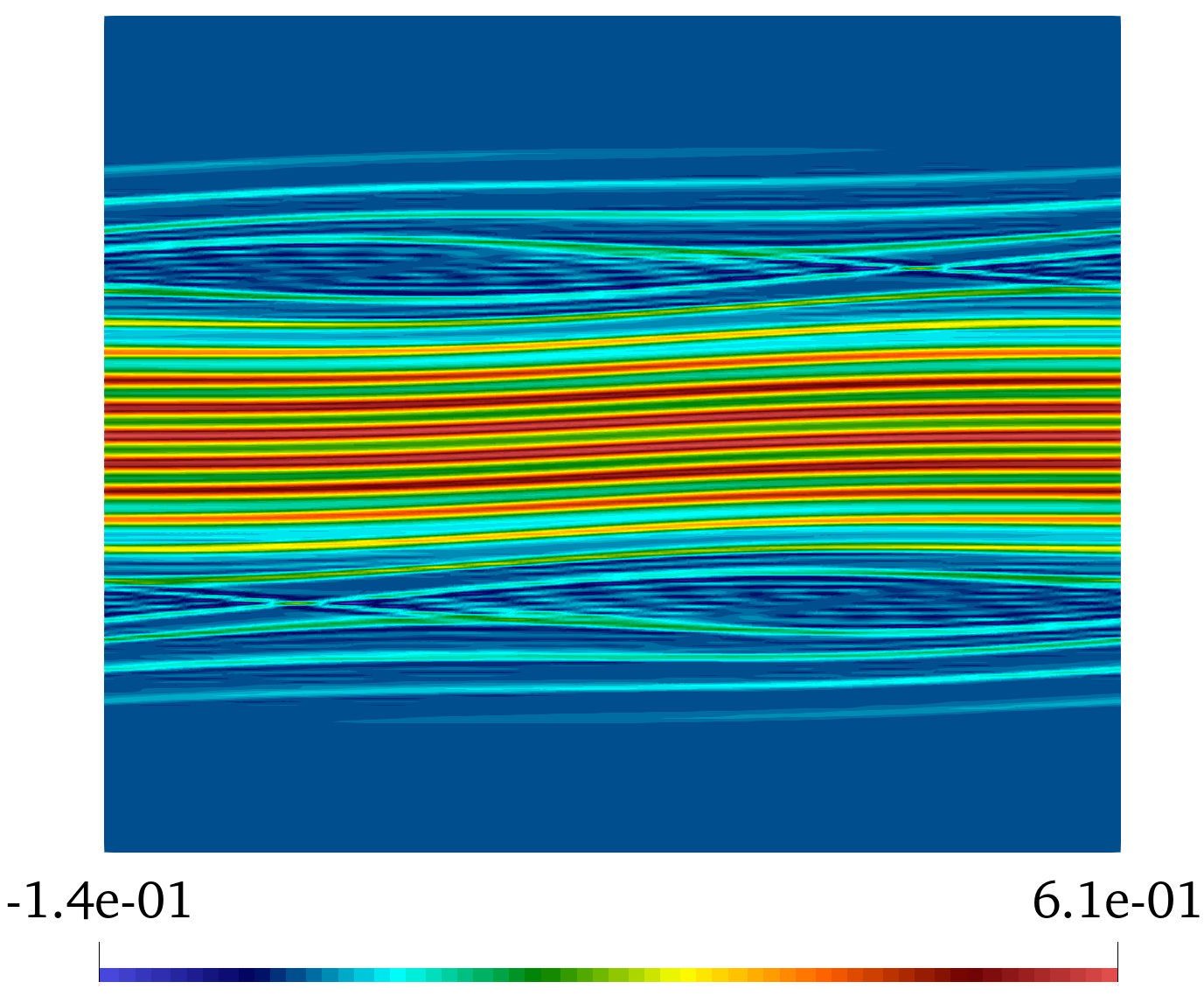}
    \caption{Galerkin}
  \end{subfigure}
  \begin{subfigure}[b]{0.32\textwidth}
    \includegraphics[width=\linewidth]{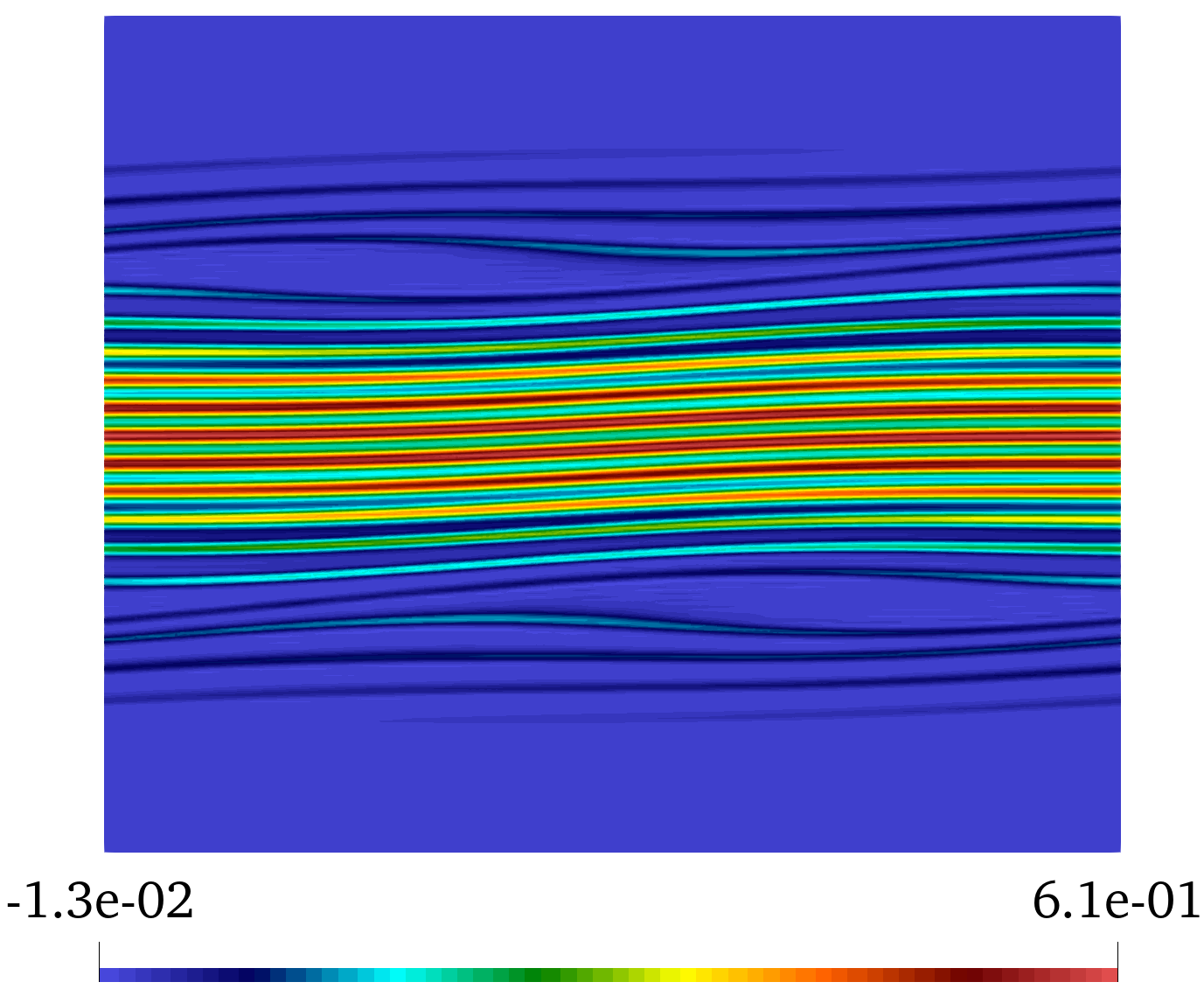}
    \caption{RV with $d_{ij}^{\mathsf{H},  {2h}}$}
  \end{subfigure}
  \begin{subfigure}[b]{0.32\textwidth}
    \includegraphics[width=\linewidth]{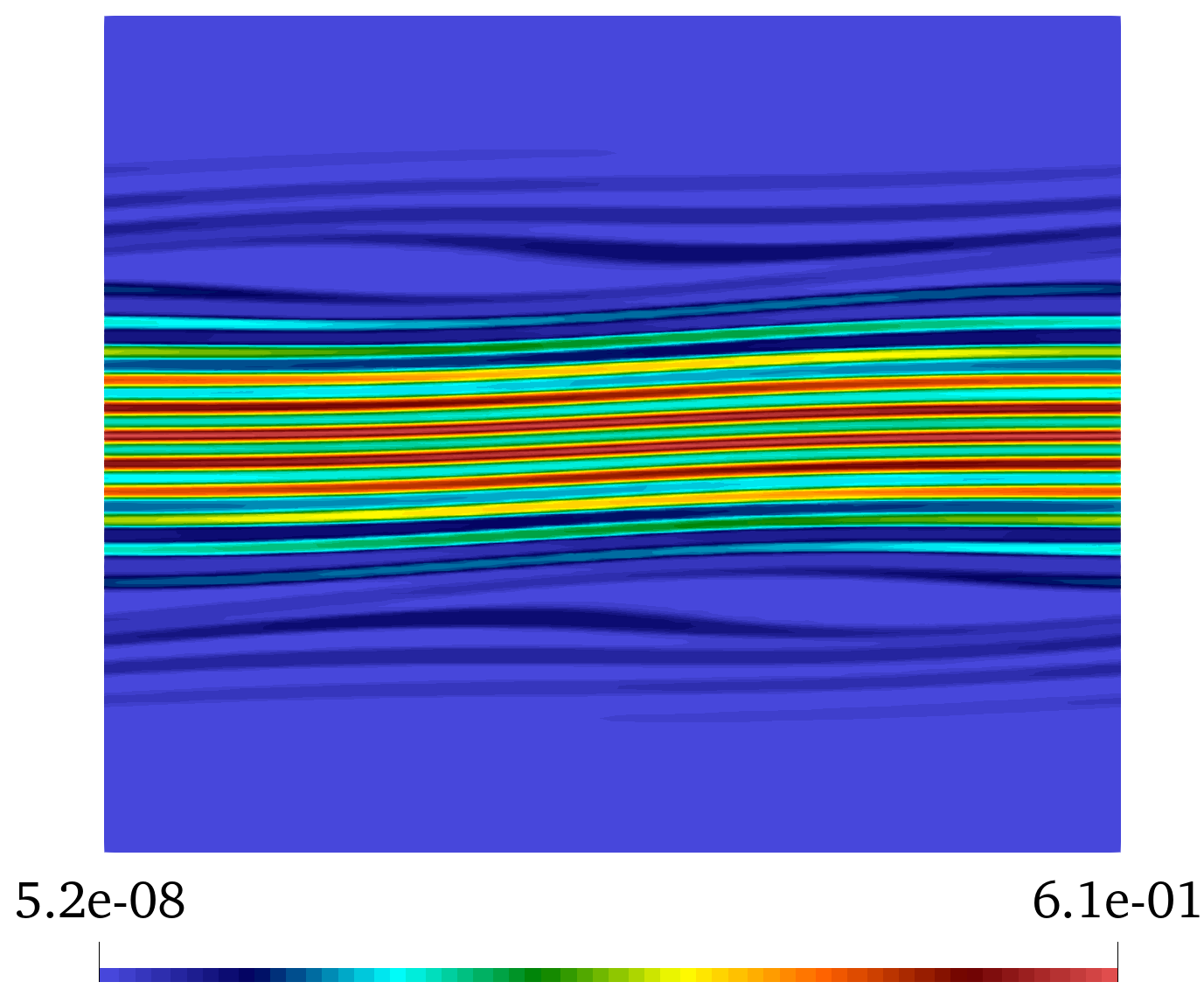}
    \caption{RV with $d_{ij}^{\mathsf{H},  {2h}}$ + limiter}
  \end{subfigure}
  \caption{Strong Landau damping: $f_h$ at $t=30$. The functions are plotted at the nodal points, and the mesh consists of \(128 \times 256\) elements. The figure compares the solutions obtained using Galerkin and residual viscosity methods without and with the limiter.}
  \label{fig:LD_t30}
\end{figure}

\begin{figure}[htbp]
  \centering
  \begin{subfigure}[b]{0.5\textwidth}
    \includegraphics[width=\linewidth]{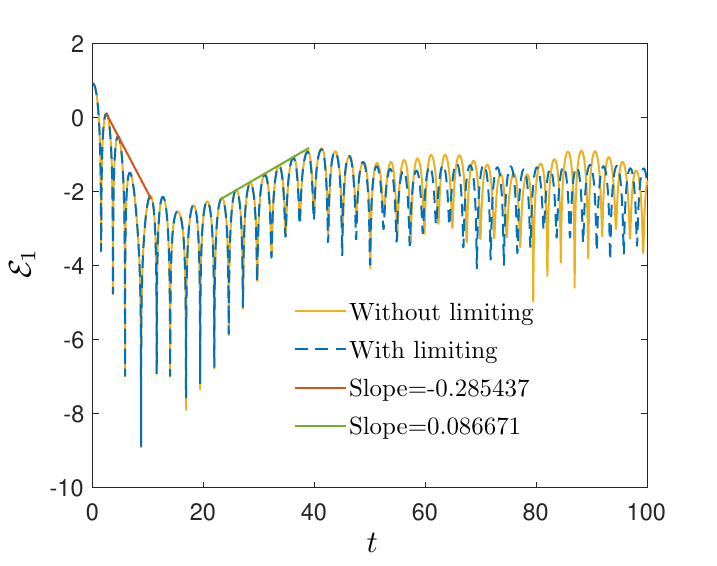}
  \end{subfigure}
  \caption{Strong Landau damping: time evolution of \mred{$\log_{10}(\mathcal{E}_1)$}. The mesh consists of \(128 \times 256\) elements. The figure compares the solutions obtained using residual viscosity without and with the limiter.}
  \label{fig:LD}
\end{figure}

\subsubsection{Two-stream instability}
Now, we consider the two-stream instability for the Vlasov--Poisson system. Let the phase space be $\Omega := [0,4\pi]\times[-5,5]$ and the initial data be
\begin{equation}
  f(x,v,0)=\frac{1}{\sqrt{2\pi}}v^2{\rm exp}\left(-\frac{v^2}{2}\right)(1+\alpha{\rm cos}(k x)), \notag
\end{equation}
where $k = 0.5$, $\alpha = 0.01$. The solutions at $t = 10$, $t = 15$, and $t = 20$ are plotted in Figure~\ref{fig:TS1} (a), and we also show the values of $f_h$ along the diagonal line of the phase space in Figure~\ref{fig:TS1} (b). The solution strictly preserves the positivity of the distribution function, and there is a lower bound slightly above zero, as shown in Figure~\ref{fig:TS1} (b).

In Figure~\ref{fig:TS_residual}, we present the residuals in which the first term in~\eqref{eq:final_residual} is computed using three coarser finite element spaces, namely $\calV_{ {2h}}$, $\calV_{ {4h}}$, and $\calV_{ {8h}}$. We observe that, for the residual $R_h^n$, the overall distribution remains qualitatively similar when different mesh resolutions are used for $R_{1, {qh}}^n \in \calV_{ {qh}}$. Moreover, we find that using a coarser space $\calV_{ {qh}}$ results in residuals of larger magnitude, which corresponds to increased numerical diffusion.
\begin{figure}[htbp]
  \centering
  \begin{subfigure}[b]{0.33\textwidth}
    \includegraphics[width=\linewidth]{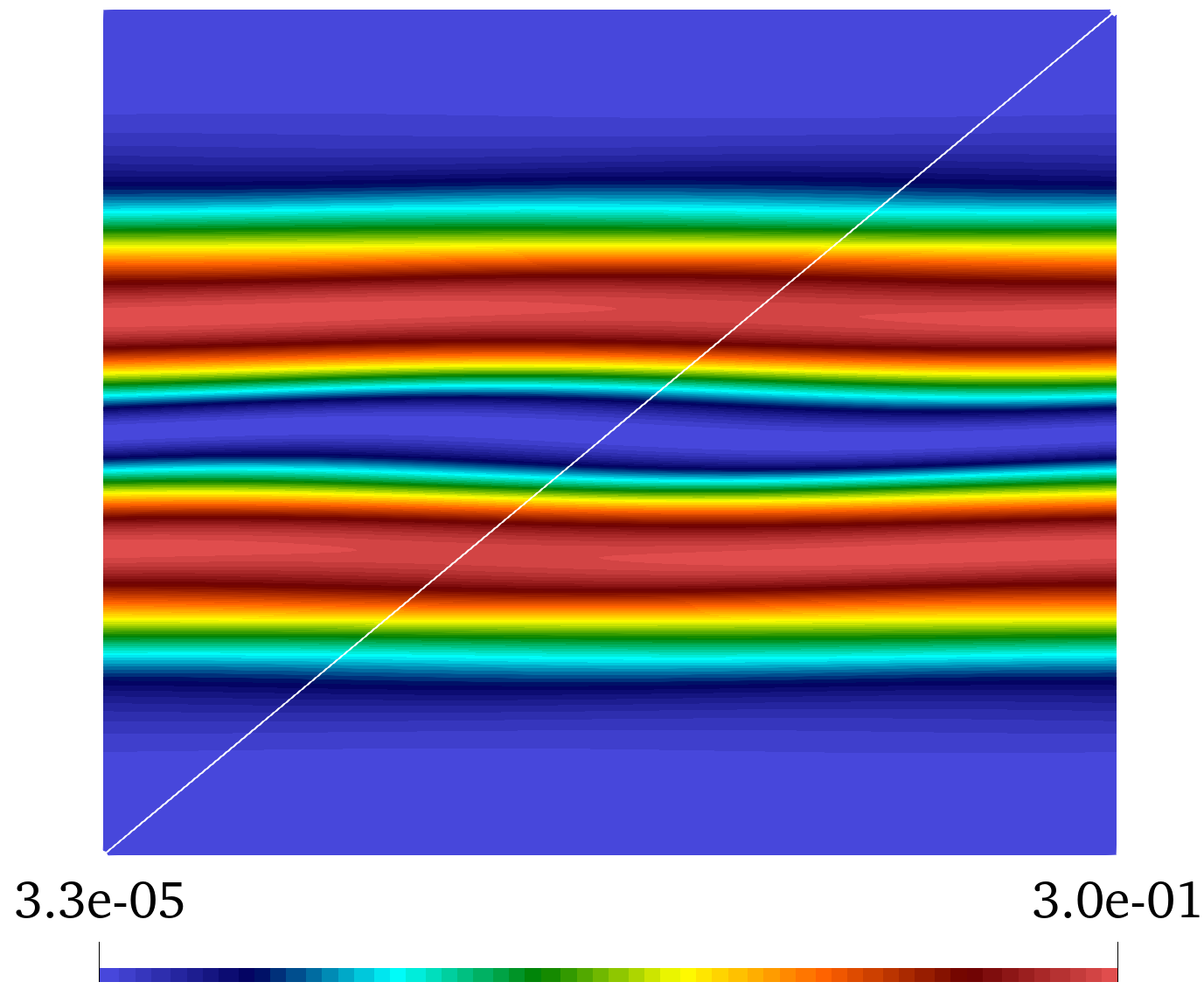}
    \caption{$f_h$ at $t=10$}
  \end{subfigure}\hfill
  \begin{subfigure}[b]{0.33\textwidth}
    \includegraphics[width=\linewidth]{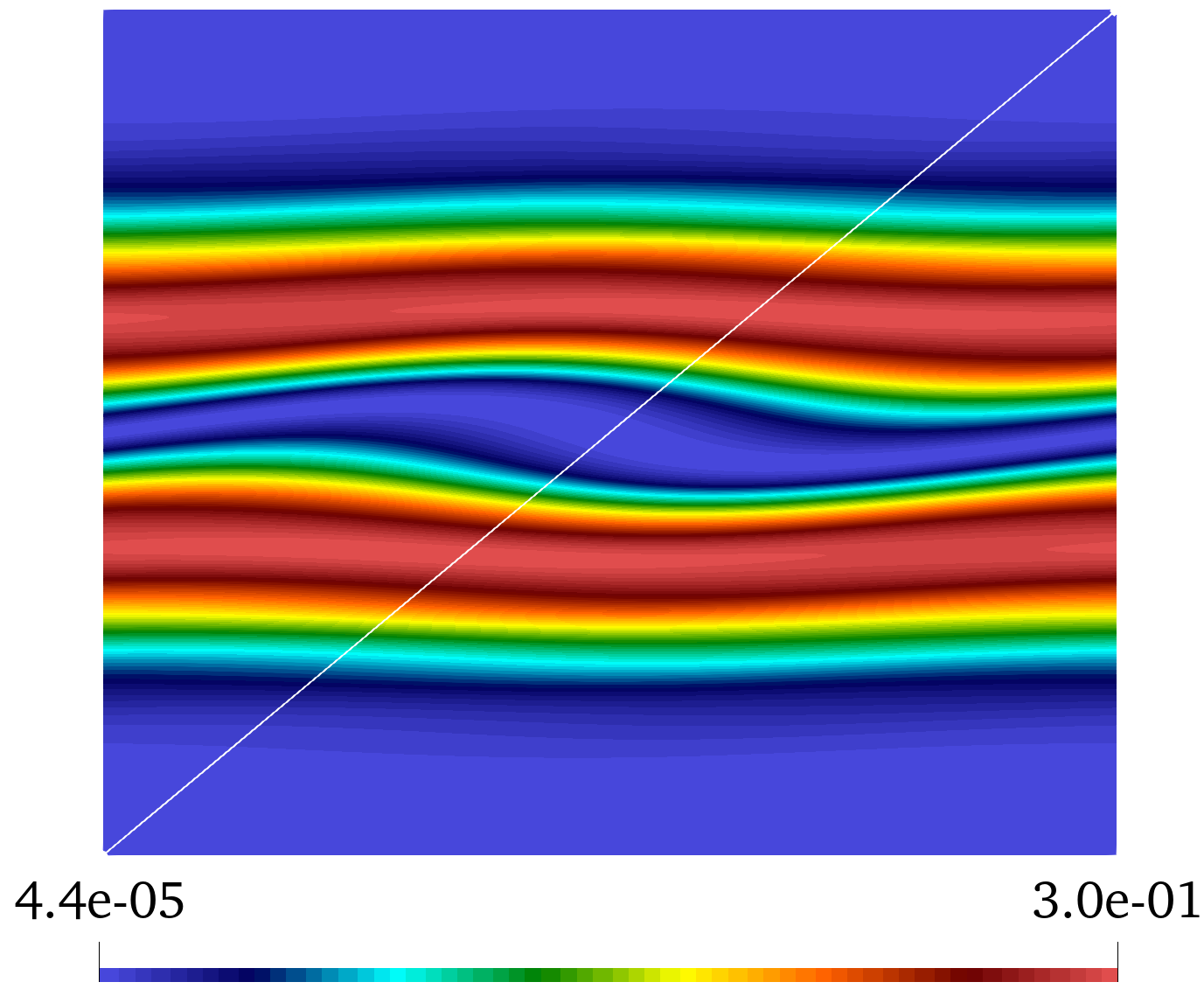}
    \caption{$f_h$ at $t=15$}
  \end{subfigure}\hfill
  \begin{subfigure}[b]{0.33\textwidth}
    \includegraphics[width=\linewidth]{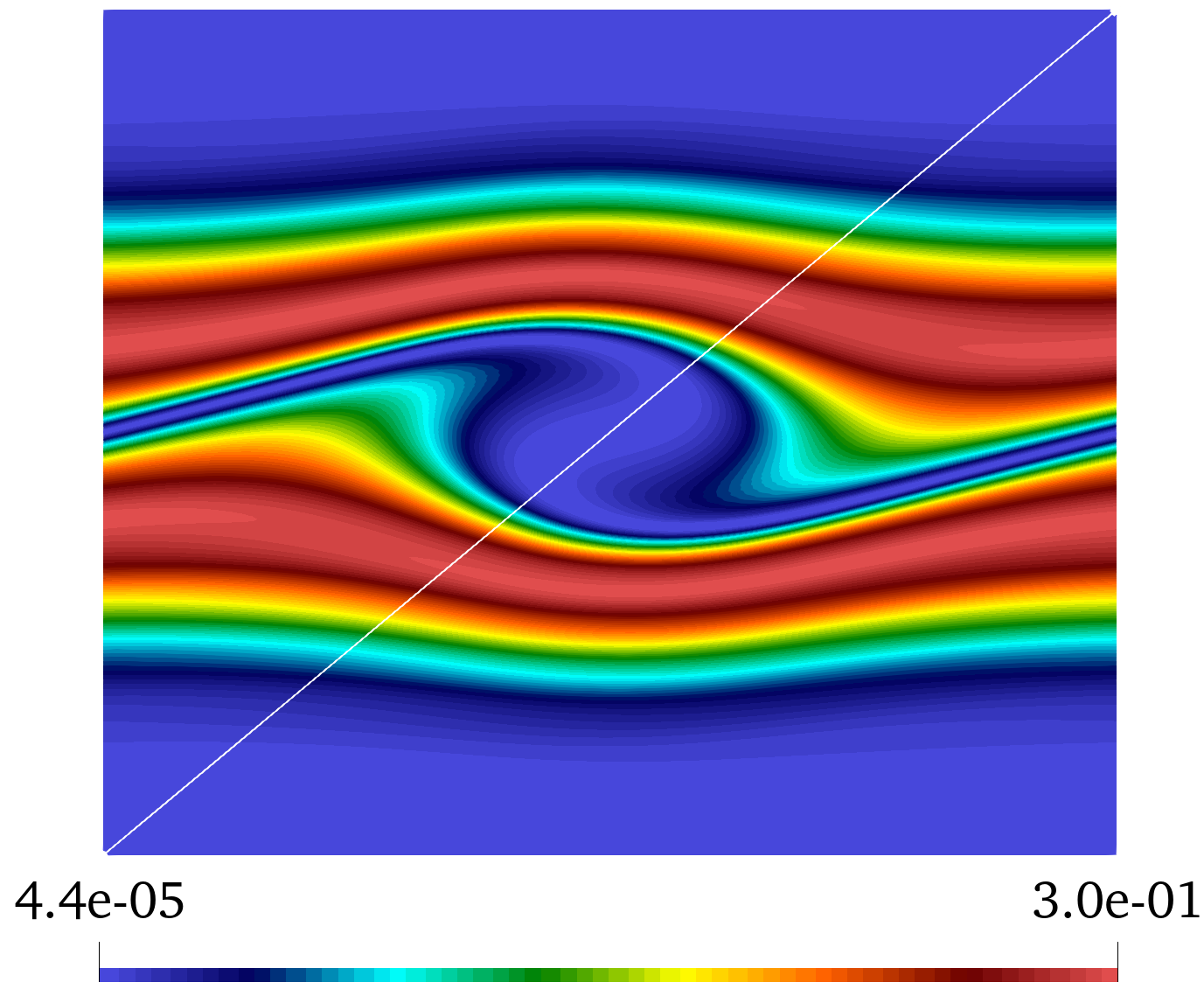}
    \caption{$f_h$ at $t=20$}
  \end{subfigure}\hfill
  \begin{subfigure}[b]{0.33\textwidth}
    \includegraphics[width=\linewidth]{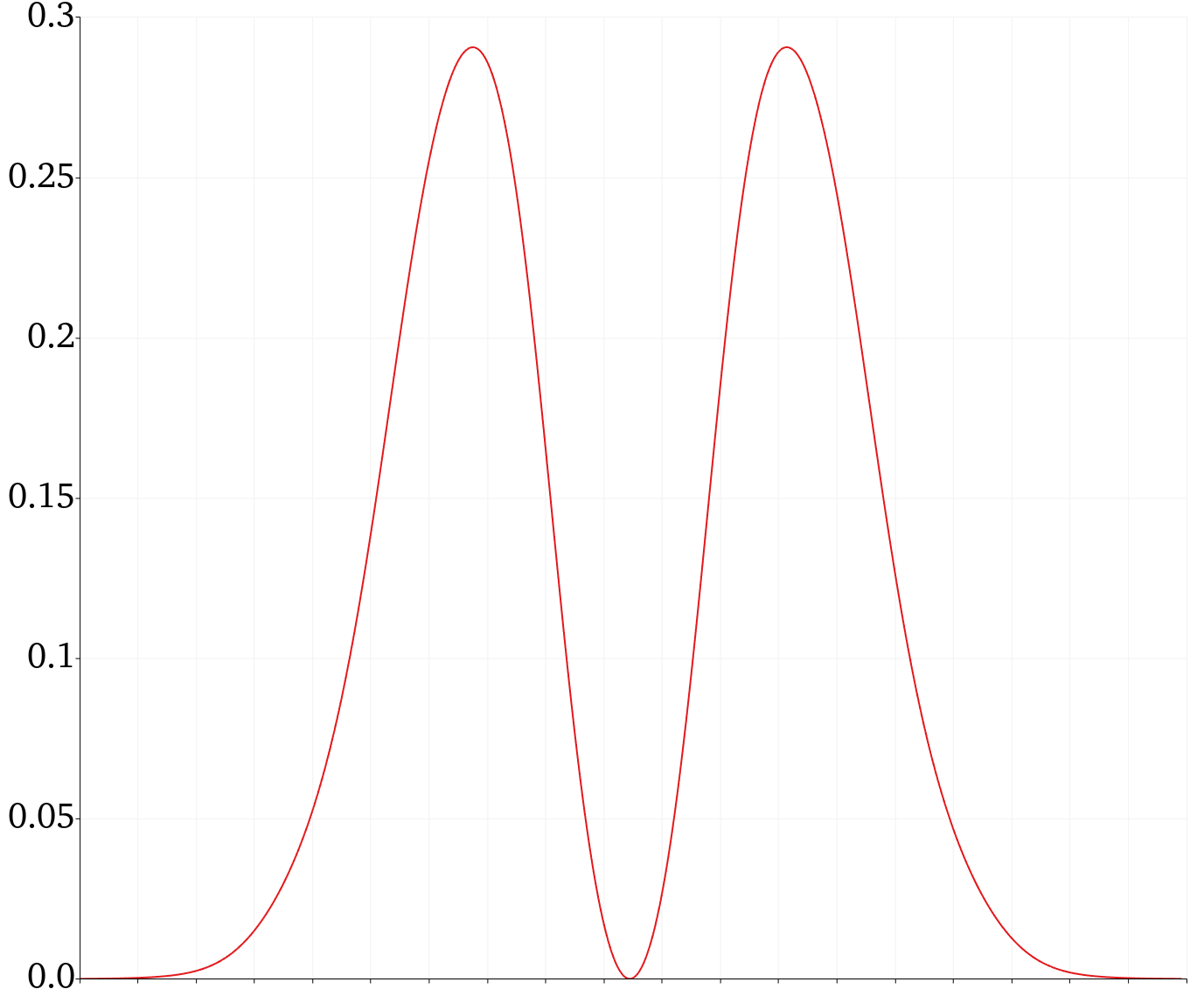}
    \caption{\mred{$t=10$}}
  \end{subfigure}\hfill
  \begin{subfigure}[b]{0.33\textwidth}
    \includegraphics[width=\linewidth]{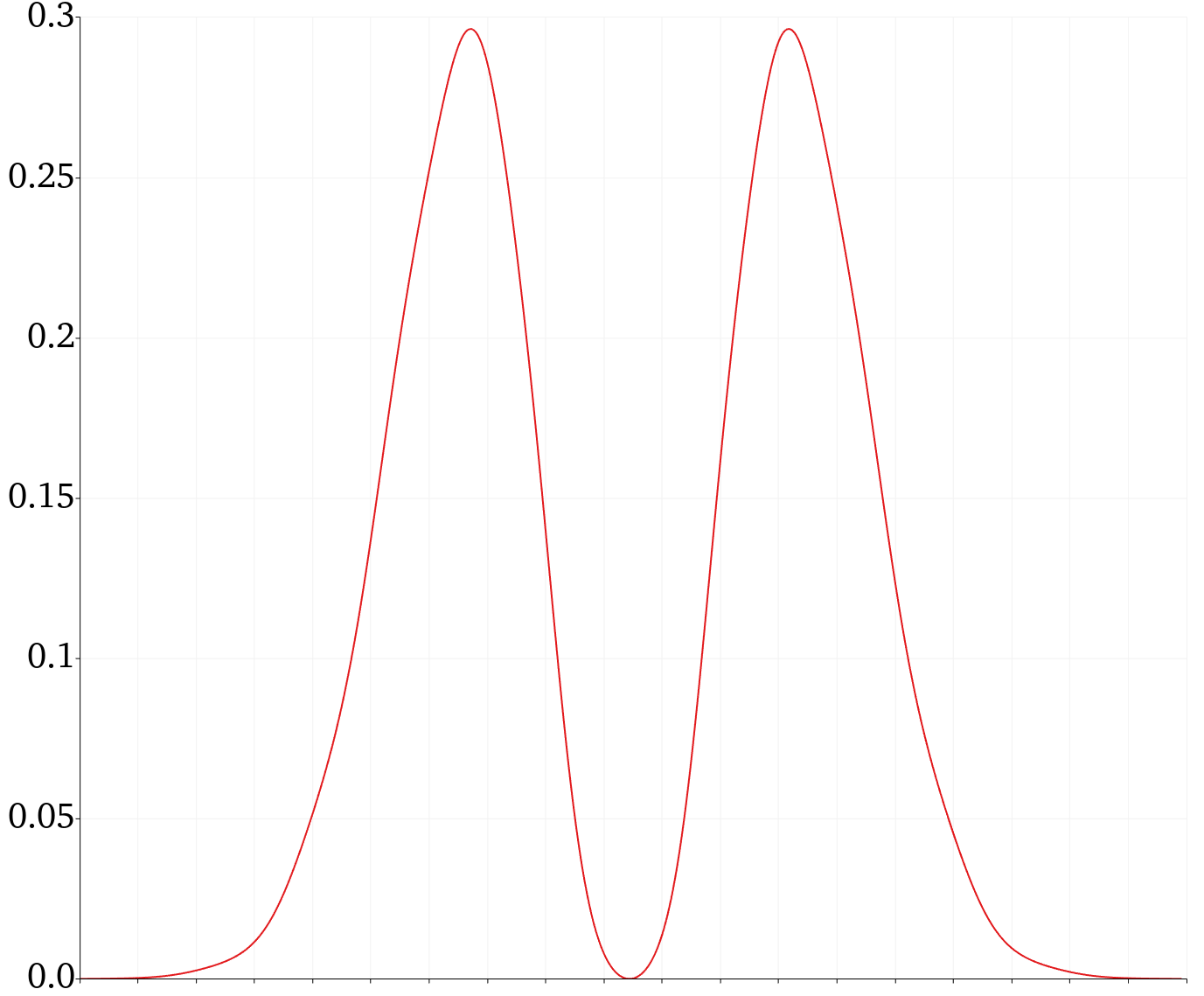}
    \caption{\mred{$t=15$}}
  \end{subfigure}\hfill
  \begin{subfigure}[b]{0.33\textwidth}
    \includegraphics[width=\linewidth]{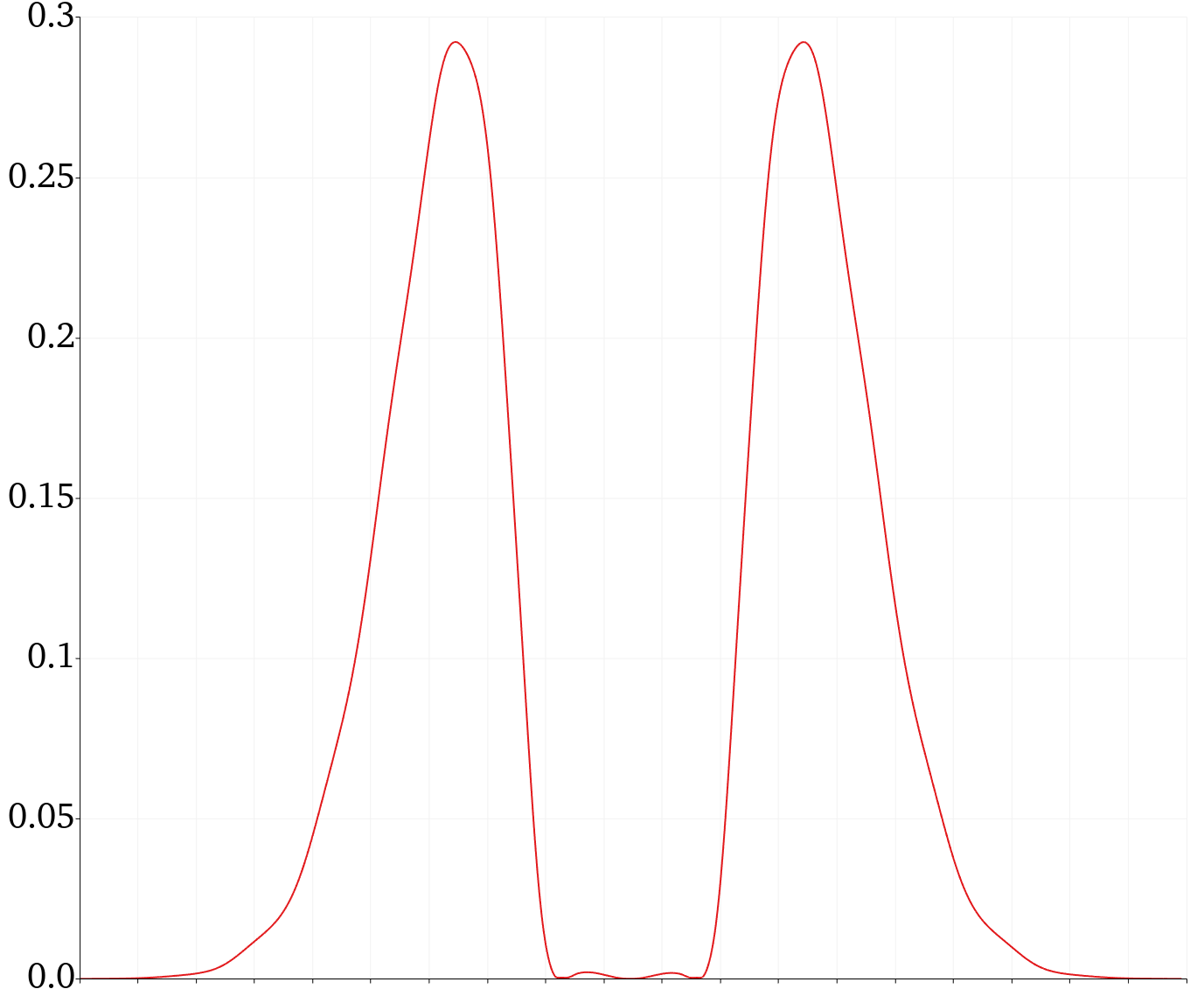}
    \caption{\mred{$t=20$}}
  \end{subfigure}\hfill
  \caption{Two-stream instability. \mred{The three columns correspond to $t=10$, $15$, and $20$, respectively. The upper row shows $f_h$ in phase space, and the lower row shows the corresponding diagonal cuts.} The functions are plotted at the nodal points, and the mesh consists of $256\times512$ elements.}
  \label{fig:TS1}
\end{figure}

\begin{figure}[htbp]
  \centering
  \begin{subfigure}[b]{0.32\textwidth}
    \includegraphics[width=\linewidth]{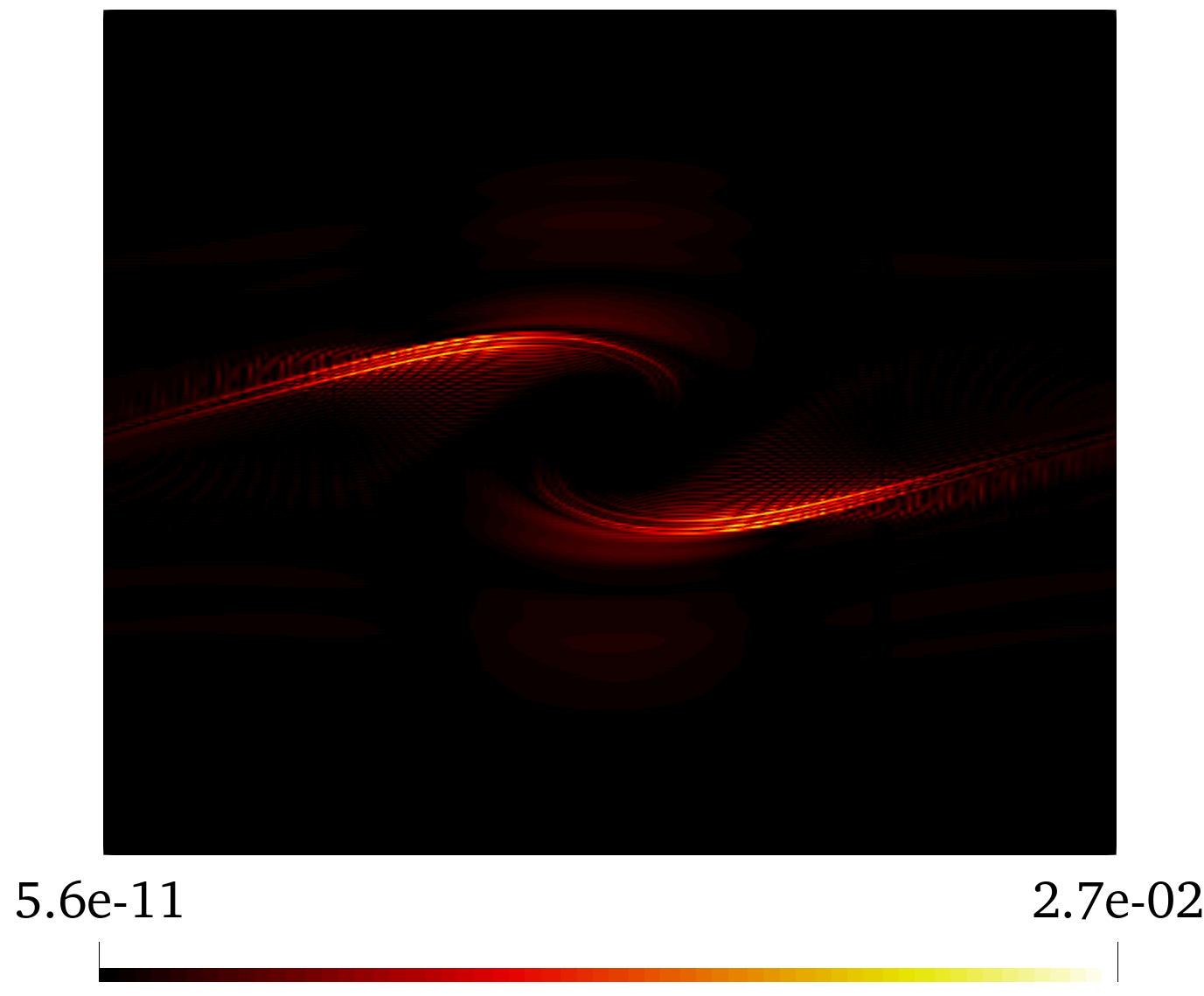}
    \caption{$I_h\big(R^n_{1, {2h}}\big)-R_{2,h}^n$}
  \end{subfigure}
  \begin{subfigure}[b]{0.32\textwidth}
    \includegraphics[width=\linewidth]{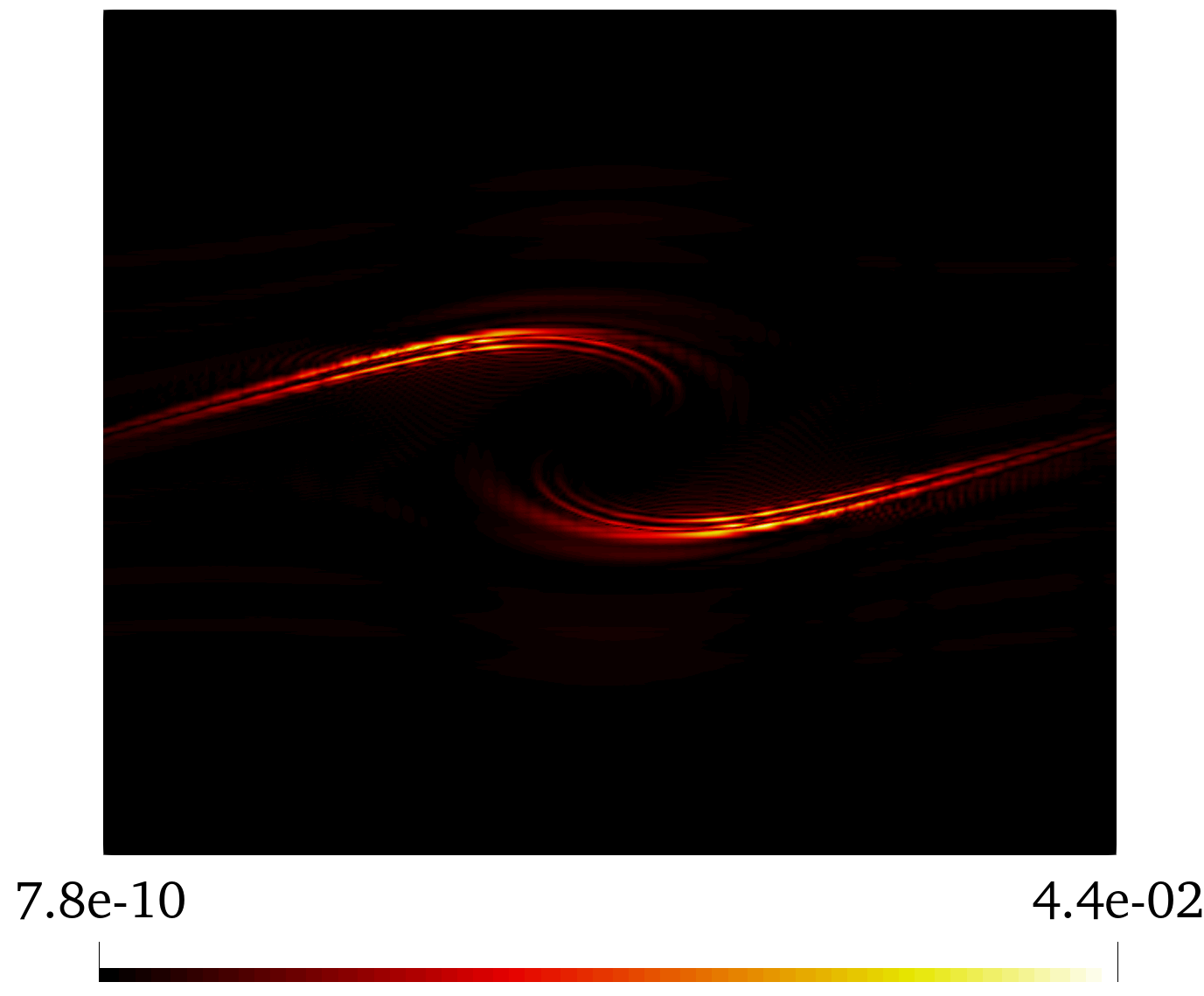}
    \caption{$I_h\big(R^n_{1, {4h}}\big)-R_{2,h}^n$}
  \end{subfigure}
  \begin{subfigure}[b]{0.32\textwidth}
    \includegraphics[width=\linewidth]{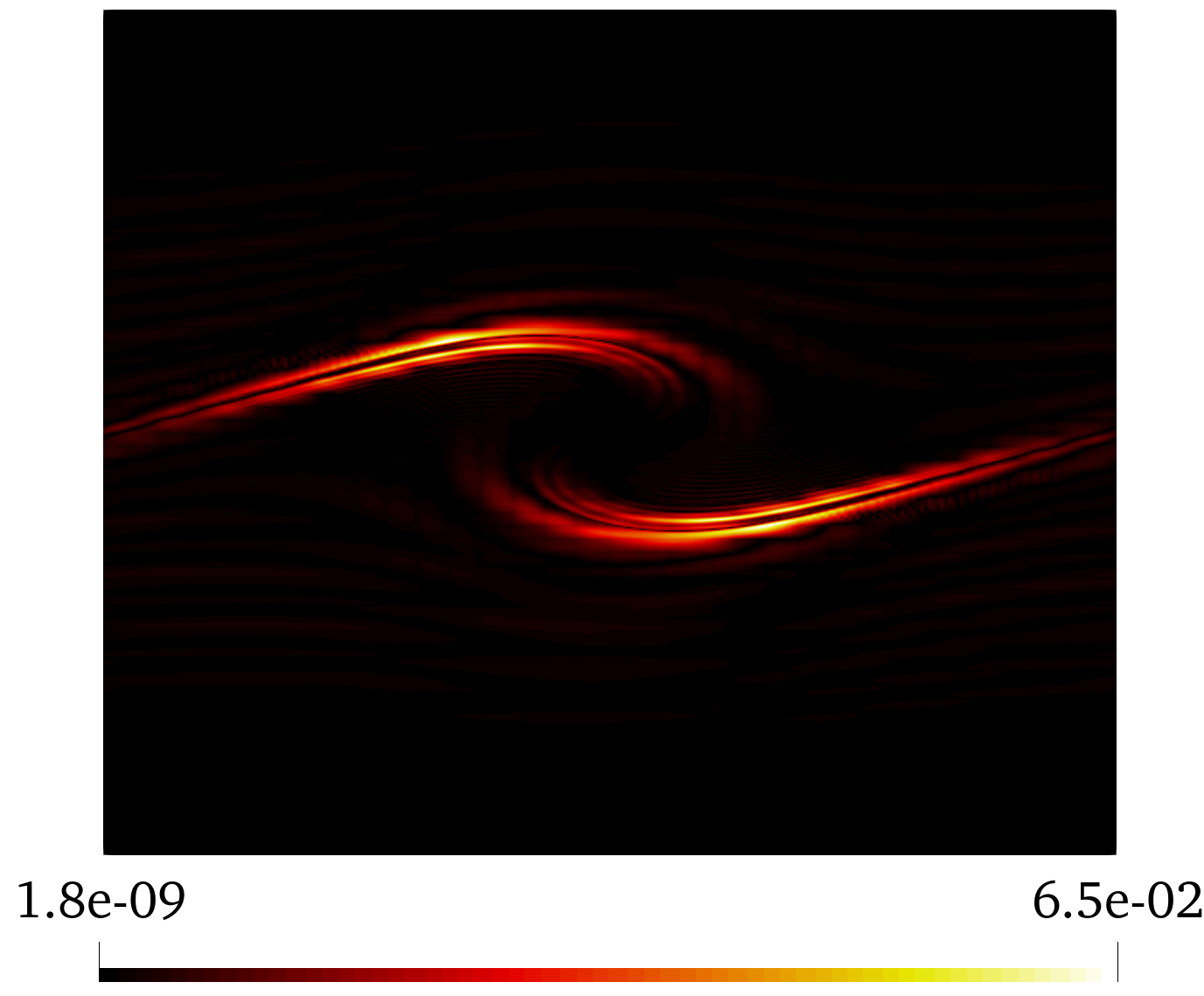}
    \caption{$I_h\big(R^n_{1, {8h}}\big)-R_{2,h}^n$}
  \end{subfigure}
  \caption{Two-stream instability: $R^n_h = I_h\big(R^n_{1, {qh}}\big)-R_{2,h}^n$ from \eqref{eq:final_residual} at $t = 20$ for $q=2,4,8$. The functions are plotted at the nodal points, and the mesh consists of $256\times512$ elements.}
  \label{fig:TS_residual}
\end{figure}

\subsection{1D2V Vlasov--Maxwell equations}
Now we apply our methods to the Vlasov--Maxwell equations in the so-called 1D2V phase space. We consider the streaming Weibel instability, focusing on the evolution of the electromagnetic energy and its characteristics, such as the growth rate. Relatively small CFL numbers are used in order to accurately capture these dynamics.

We now consider the streaming Weibel instability and take the initial data
\begin{equation}\notag
  \begin{aligned}
    f(x, v_1, v_2, 0) =& \frac{1}{2\pi\sigma^2}{\rm exp}\left(-\frac{v_1^2}{2\sigma^2}\right)\left(\delta{\rm exp}\left(-\frac{(v_2-v_{0,1})^2}{2\sigma^2}\right)\right.\\
                         &\left.+(1-\delta){\rm exp}\left(-\frac{(v_2-v_{0,2})^2}{2\sigma^2}\right)\right),\\
    B_3(x, 0) = &10^{-3}{\rm sin}(k x),\\
    E_2(x, 0) = &0,
  \end{aligned}
\end{equation}
and $E_1(x, 0)$ is computed by solving the Poisson equation.
The parameters are: $\sigma = 0.1/\sqrt{2}$, $k = 0.2$, $v_{0,1}=0.5$, $v_{0,2}=-0.1$ and $\delta=1/6$. The phase space is $\Omega=[0,2\pi/k]\times[-5\sigma,5\sigma]\times[-1.2,1.2]$. The CFL number is set to  {0.2}.

In Figure~\ref{fig:SWI}, we compare the results obtained with and without divergence cleaning. In both cases, the solutions show good agreement with the growth rate  {0.03} reported in \cite{Kraus_Kormann_Morrison_Sonnendruecker_2017}. The corresponding Gauss’s law errors are shown in Figure~\ref{fig:SWI_error}, where we observe that the divergence-cleaning technique significantly reduces the constraint violation. Hence, we see that the divergence-cleaning method proposed in this manuscript restores the divergence constraint on $\pmb{E}_h$ without destroying the results describing the electromagnetic fields. 
We also observe that, even with divergence cleaning, a small drift persists in the Gauss error. This drift arises from the numerical solution of the Poisson equation \eqref{eq:dPhin} using finite element methods, which introduces a small error in the divergence-cleaning step. This error accumulates over time and leads to the drift observed in Figure~\ref{fig:SWI_error}.

\begin{figure}[htbp]
  \centering
  \begin{subfigure}[b]{0.45\textwidth}
    \includegraphics[width=\linewidth]{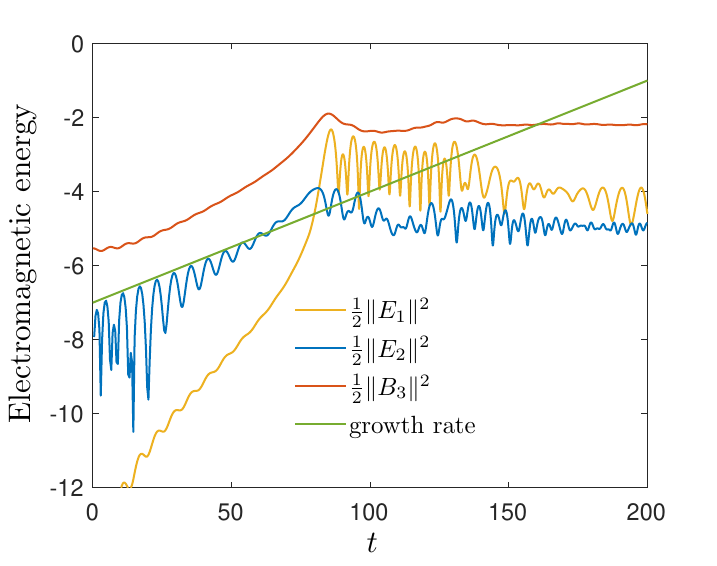}
    \caption{With divergence cleaning}
  \end{subfigure}~
  \begin{subfigure}[b]{0.45\textwidth}
    \includegraphics[width=\linewidth]{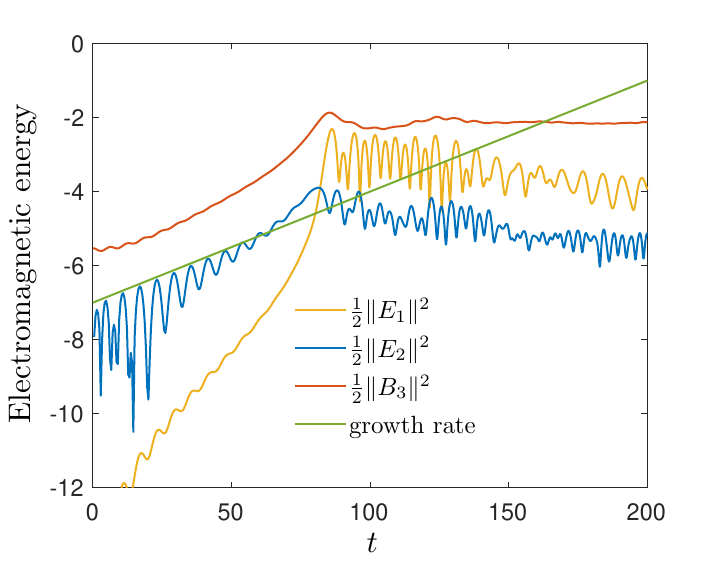}
    \caption{Without divergence cleaning}
  \end{subfigure}
  \caption{Streaming Weibel instability: the two electric and the magnetic energy together with the analytic growth rate. The mesh consists of $32\times64^2$ elements.}
  \label{fig:SWI}
\end{figure}
\begin{figure}[htbp]
  \centering
  \begin{subfigure}[b]{0.45\textwidth}
    \includegraphics[width=\linewidth]{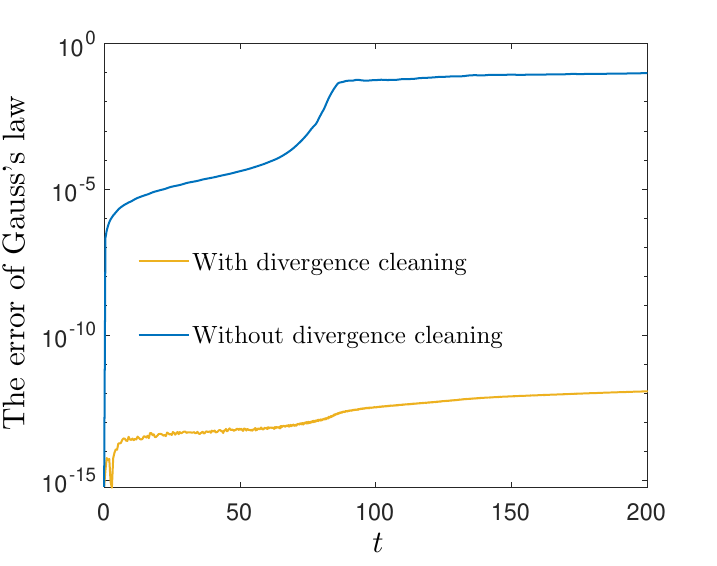}
  \end{subfigure}
  \caption{Streaming Weibel instability: errors in Gauss's law, with and without divergence cleaning.}
  \label{fig:SWI_error}
\end{figure}

\subsection{2D2V \mred{diocotron} instability}
Finally, we apply our methods to the 2D2V diocotron instability. We consider the Vlasov–Poisson equation with an external magnetic field $\pmb{B}_{\rm ext}$:
\begin{equation}
  {\partial_t f}+\pmb{v}\cdot\nabla_{\pmb{x}}f+\left(\pmb{E}+\pmb{v}\times\pmb{B}_{\rm ext}\right)\cdot\nabla_{\pmb{{v}}}f=0.\notag
\end{equation}
Taking the same initial data as our previous work \cite{MR4945433}:
\begin{equation}
  f(\pmb{x},\pmb{v},0)=\frac{\rho_0(\pmb{x})}{2\pi}\exp\left(-\frac{v_1^2+v_2^2}{2}\right),\notag
\end{equation}
where
\begin{equation}
  \rho_0(\pmb{x})=\left\{
    \begin{aligned}
      &(1+\alpha{\rm cos}(k\theta))\exp(-4(\Vert\pmb{x}\Vert-6.5)^2),&\qquad &{\rm if}\ 5\leq\Vert\pmb{x}\Vert\leq8,\\
      &0,&\qquad &{\rm otherwise},
    \end{aligned}
  \right.\notag
\end{equation}
where $k = 6$, $\theta = {\rm atan}(x_2/x_1)$, and $\alpha = 0.2$. The phase space is $\Omega:=\Omega_{\pmb{x}}\times\Omega_{\pmb{v}}$, where $\Omega_{\pmb{x}}:=\{\pmb{x}\in\polR^2:\Vert\pmb{x}\Vert\leq12\}$ and $\Omega_{\pmb{v}}:=[-5,5]^2$. In this simulation, we use $\pmb{B}_{\rm ext}=(0,0,1)$; moreover, the electric field is obtained by solving
  \begin{equation}
    \pmb{E}=-\nabla_{\pmb{x}}\Phi,\qquad-\nabla_{\pmb{x}}^2\Phi=\int_{\pmb{\Omega_{\pmb{v}}}}f\ {\rm d}\pmb{v}.
  \end{equation}

In phase space, we use 10000 $\polP_1$ elements in the physical domain and $20^2$ $\polQ_1$ elements in the velocity domain. When computing the residual, $\tilde{\calV}_h$ consists of $2500\ \polP_1\times10^2\ \polQ_1$ elements. Figure~\ref{fig:DI} shows the distribution of the charge density at $t=0$ and $1$. The distribution of charge density closely resembles the result reported in \cite[Fig. 9]{MR4945433}. Throughout the simulation, however, the charge density remains strictly positive, indicating that the positivity-preserving scheme successfully maintains non-negativity of the distribution function even in higher-dimensional settings.

\begin{figure}[htbp] 
  \centering
  \begin{subfigure}[b]{0.4\textwidth}
    \includegraphics[width=\linewidth]{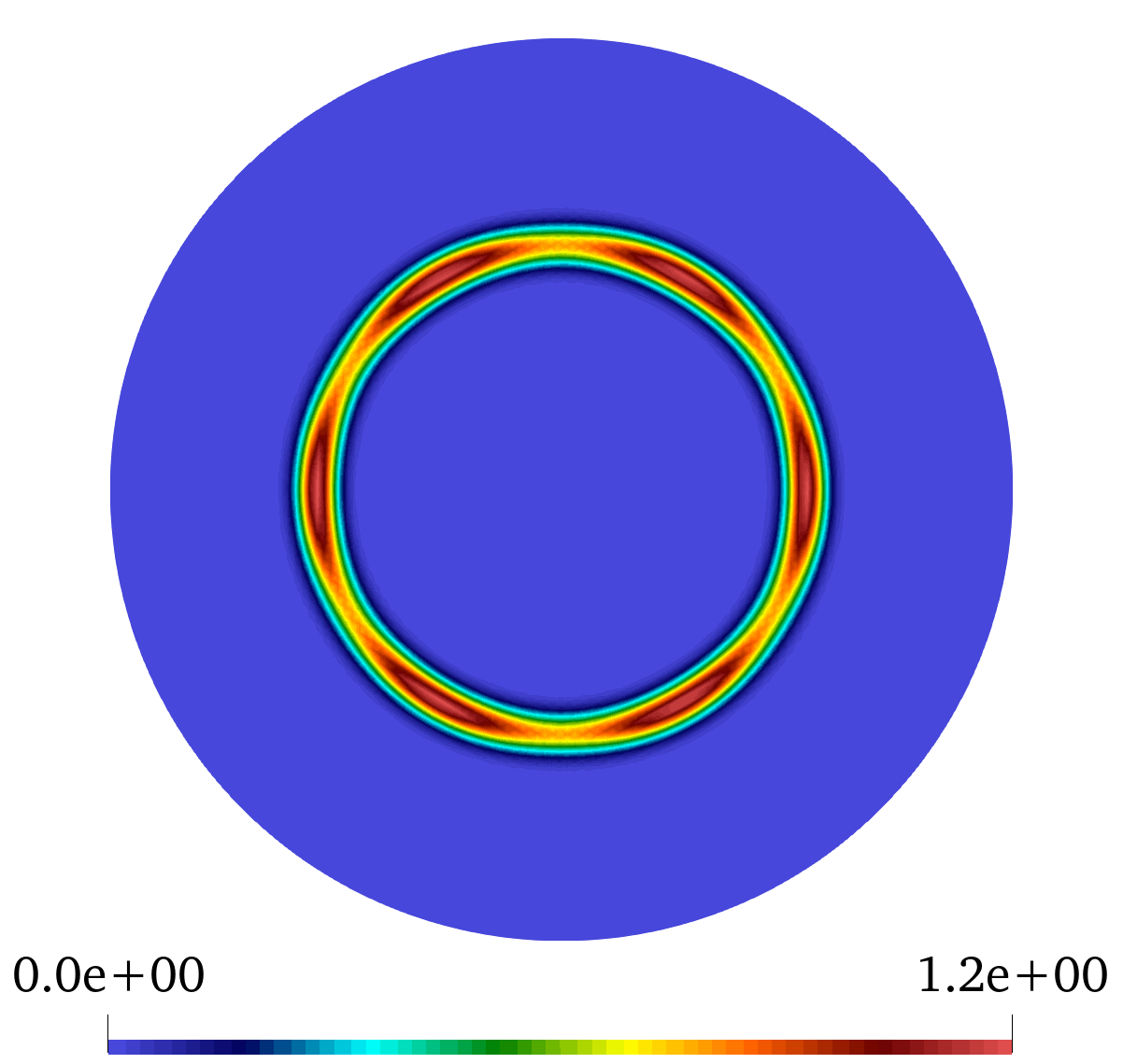}
    \caption{$t=0$}
  \end{subfigure}\quad
  \begin{subfigure}[b]{0.4\textwidth}
    \includegraphics[width=\linewidth]{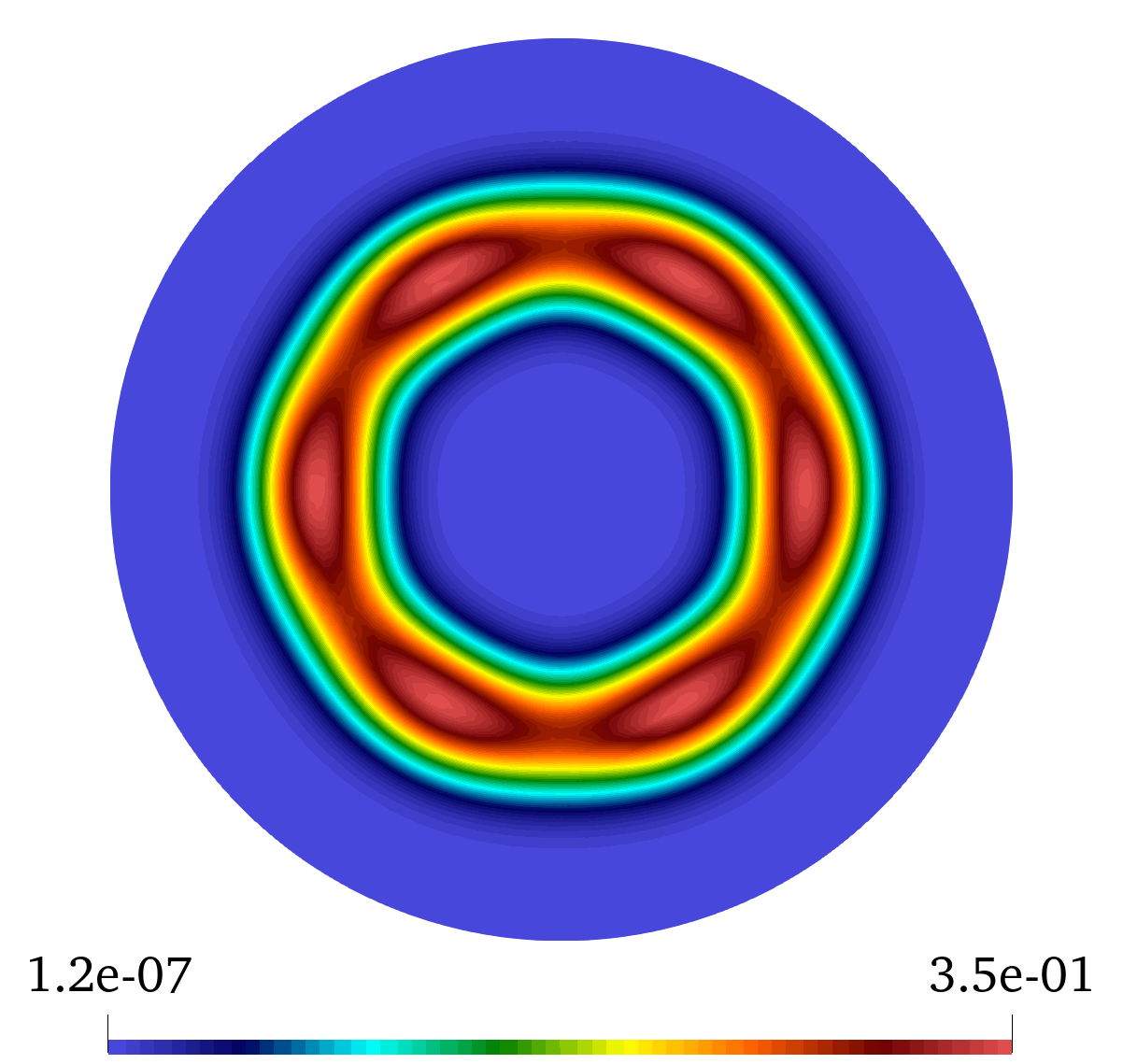}
    \caption{$t=1$}
  \end{subfigure}
  \caption{\mred{Diocotron} instability: $\rho_h$ at $t = 0$ and $1$. The functions are plotted at the nodal points, and the mesh consists of $10000\ \polP_1\times20^2\ \polQ_1$ elements.}
  \label{fig:DI}
\end{figure}

\section{Conclusion}\label{section:conclu}
In this work, we have introduced a high-order po\-si\-ti\-vi\-ty-pr\-es\-er\-vi\-ng numerical scheme for the Vlasov equations based on a continuous finite element discretization. The method combines a low-order scheme that satisfies a discrete maximum principle with a high-order scheme through a convex limiting strategy. Numerical results demonstrate that the proposed scheme preserves the positivity of both the distribution function and the charge density, and achieves second-order accuracy for smooth solutions when first-order polynomial approximations are employed.

Although the method has been successfully applied to four-dimensional problems, it becomes computationally expensive for higher polynomial degrees and for full six-dimensional simulations. Ongoing work aims to address these challenges, and the results will be reported in future publications.

\section*{Acknowledgments}
Some computations were performed on the UPPMAX supercomputing resources under project UPPMAX 2025/2-269.

\bibliographystyle{unsrturl}
\bibliography{ref_arxiv_bibtex}

\end{document}